\documentclass[11pt,a4paper]{article}
\usepackage[shortlabels]{enumitem}
\usepackage[margin=2.5cm]{geometry}
\usepackage{amsmath,amssymb,amsthm}
\usepackage{mathrsfs}
\usepackage{enumitem}
\usepackage{hyperref}
\usepackage{cleveref}
\usepackage{booktabs}

\newtheorem{theorem}{Theorem}[section]
\newtheorem{lemma}[theorem]{Lemma}

\theoremstyle{definition}

\newtheorem{remark}[theorem]{Remark}
\newtheorem{assumption}[theorem]{Assumption}

\newcommand{\dt}{\tilde}

\title{Well-posedness and Stability of Thermoelastic Suspension Bridges with Cattaneo Heat Conduction: The Role of Viscoelastic Memory}

\author{Jun Zhou\footnote{Corresponding author. Email: jzhou@swu.edu.cn} ~~~~~~ Jiamin He\footnote{Email: 1551937477@qq.com}\\\\{School of Mathematics and Statistics, Southwest University},\\
		{Chongqing 400715, People's Republic of China}}
\date{}
\allowdisplaybreaks
\begin{document}
\maketitle

\begin{abstract}
Thermoelastic suspension bridges operate under strong thermo-mechanical coupling, and the classical Fourier law fails to capture the finite speed of thermal propagation observed in low-conductivity materials such as concrete. This paper investigates a suspension bridge model in which the heat flux is governed by the Cattaneo law with relaxation time $\tau$, and in which the bridge deck may or may not possess viscoelastic memory. Using operator semigroup theory, we establish well-posedness and derive sharp decay rates for both regimes. For the memoryless system ($\delta=0$), we show that the semigroup is polynomially stable with rate $t^{-1/2}$ when the structural parameters satisfy $\frac{\rho_1}{\kappa} = \frac{\rho_2}{b}$, and that the rate deteriorates to $t^{-1/4}$ otherwise; moreover, exponential stability fails whenever $\chi_0 \neq 0$, or $\chi_0 = 0$ and $\chi_1 = 0$, where $\chi_0$ and $\chi_1$ are parameter-dependent coefficients. In sharp contrast, when viscoelastic memory is present ($\delta=1$), the semigroup becomes exponentially stable, regardless of the structural parameter ratio. This reveals a regularizing role of viscoelastic memory: it converts polynomial decay into exponential decay and removes the parameter-ratio obstruction. The results provide a theoretical basis for exploiting inherent viscoelasticity in long-span bridges and for optimizing deck parameters in memoryless structures.
\end{abstract}


\section{Introduction}
\label{sec:introduction}

Suspension bridges are among the most elegant and efficient long-span structures, yet their mechanical behavior under combined environmental and operational loads is highly complex. Accurate mathematical modeling is therefore essential for design, safety assessment, and long-term maintenance. In a typical suspension bridge model, the main cable is idealized as an elastic string that captures vertical vibrations, while the bridge deck is described by Timoshenko beam theory \cite{timoshenko1921lxvi}. The main cable and the deck are connected through linear elastic springs representing the suspenders, giving rise to a strongly coupled interactive system. To analyze the stability of such a system, damping mechanisms---such as internal friction or thermoelastic effects---are often incorporated.

In this paper, we consider the following suspension bridge system in $ (0,L) \times \mathbb{R}^+$, in which the Cattaneo heat conduction law plays a key role:
\begin{align}\label{model1.1}
	\begin{cases}
		\rho u_{tt} - \alpha u_{xx} - \beta(\varphi - u) + \gamma u_t = 0, \\
		\rho_1 \varphi_{tt} - \kappa(\varphi_x + \psi)_x + \beta (\varphi - u) + m \theta_x = 0, \\
		\rho_2 \psi_{tt} - b \psi_{xx} + \kappa(\varphi_x + \psi) + \delta \int_0^\infty g(s) \psi_{xx}(t-s)  ds - m \theta = 0, \\
		\rho_3 \theta_t + q_x + m (\varphi_x + \psi)_t = 0, \\
		\tau q_t + \sigma q + \theta_x = 0,
	\end{cases}
\end{align}
where the parameter $\delta$ takes two distinct values to represent different viscoelastic conditions: $\delta = 0$ corresponds to the case \textit{without viscoelastic memory}; $\delta = 1$ corresponds to the case \textit{with viscoelastic memory}. This switchable parameter allows the model to capture both temporary and long-term memory effects that may arise from temperature fluctuations or material creep in practical suspension bridges.

The system is subject to the following homogeneous boundary conditions, which reflect the physical constraints at the two ends of the bridge ($x=0$ and $x=L$):
\begin{align*}
	&u_x(0,t) = u_x(L,t) = \varphi_x(0,t) = \varphi_x(L,t) =0\\
 &\psi(0,t) = \psi(L,t) = \theta(0,t) = \theta(L,t) = 0.
\end{align*}

The key physical quantities in the above system are defined according to their roles in describing the dynamic and thermal behavior of the suspension bridge: $u$ represents the vertical displacement of the main cable; $\varphi$ represents the transverse displacement of the bridge deck beam's cross-section; $\psi$ represents the rotation angle of the bridge deck beam's cross-section; $q$ represents the heat flux; and $\theta$ represents the temperature deviation relative to the ambient temperature.

Each parameter in the model is associated with specific material properties, structural stiffness, or physical effects of the suspension bridge components. They are classified as follows for clarity: $\rho$ denotes the density of the main cable; $\alpha$ denotes the elastic modulus of the main cable; $\gamma$ denotes the internal damping of the main cable; $\rho_1$ denotes the density of the bridge deck beam; $\kappa$ denotes the shear stiffness of the bridge deck beam; $\rho_2$ denotes the rotational moment of inertia of the bridge deck beam's cross-section; $b$ denotes the bending stiffness of the bridge deck beam; $\beta$ denotes the stiffness of the suspenders; $m$ denotes the material-dependent coupling coefficient; $\rho_3$ denotes the thermal inertia; $\tau$ denotes the relaxation time; and $\sigma$ denotes the thermal conductivity.

To ensure the physical rationality and mathematical well-posedness of the model, we impose the following assumptions on the parameters: $$\rho, \rho_1, \rho_2, \rho_3, \alpha, \kappa, b, \tau, \sigma, \beta, \gamma > 0;$$ the thermo-mechanical coupling coefficient $m$ is a nonzero real number, i.e., $m \in \mathbb{R} \setminus \{0\}$. The specific definition of the memory kernel $g(s)$, which describes the viscoelastic memory effect when $\delta=1$, will be provided in subsequent sections, together with the corresponding regularity conditions required for the solvability of the integral term.

\medskip
\noindent\textbf{Related work and motivation.}
Thermoelastic suspension bridge models have attracted considerable attention in recent decades. Raposo et al. \cite{MR4688212} studied a suspension bridge model with three types of internal damping and proved exponential stability relying solely on frictional damping. Bochicchio et al. \cite{MR4149715} investigated a thermoelastically coupled suspension bridge model with Fourier heat conduction and established exponential stability under the combined action of internal damping and thermal diffusion. Miranda et al. \cite{MR4887867} analyzed a thermoelastic suspension bridge system with thermal dissipation acting on the shear force; they showed that the system is exponentially stable when $\frac{\kappa}{\rho_1} = \frac{b}{\rho_2}$, and otherwise only polynomially stable with a decay rate of $t^{-1/2}$. For type III thermoelasticity, Chabekh et al. \cite{MR4848548} proved exponential decay of the energy. Mukiawa et al. \cite{MR4539428} considered a thermoelastic Timoshenko beam with suspenders, weakly nonlinear internal damping, and time-varying delay, and derived a general stability estimate. When viscoelastic memory effects are taken into account, Mukiawa et al. \cite{MR4913141} proposed a coupled suspension bridge system with material memory and Gurtin--Pipkin heat conduction, and obtained either exponential or polynomial decay depending on the parameter $p$. Al-Gharabli et al. \cite{MR4891851} studied a one-dimensional coupled suspension bridge system with viscoelastic damping and frictional damping modulated by a time-dependent coefficient, and derived an energy estimate. Afilal et al. \cite{MR4342054} investigated a one-dimensional coupled suspension bridge system with viscoelastic memory and obtained a decay estimate via the energy method. Further studies on the asymptotic behavior of suspension bridges can be found in \cite{MR4875485,MR3023261,MR2142424,MR2514725} for memoryless models and in \cite{MR4886591,MR3575770,MR3240285,MR3458828,MR3609015,MR4175206} for models with memory.

Despite these advances, three critical gaps remain in the literature. First, most prior thermoelastically coupled suspension bridge models adopt Fourier's law, which assumes instantaneous thermal propagation---an idealization that deviates from engineering reality, where heat conduction inherently exhibits time delay, especially in low-conductivity materials such as concrete used in bridge decks. Second, existing models either exclude memory effects entirely or treat memory as a permanent feature, whereas practical suspension bridges may experience varying viscoelastic behavior, e.g., temporary memory effects due to temperature fluctuations or long-term memory from material creep. Third, and most importantly, the combined effects of Cattaneo heat conduction and switchable viscoelastic memory on the stability of suspension bridges have not been systematically analyzed. In particular, the question of whether viscoelastic memory can regularize the dynamics---by converting polynomial decay into exponential decay and by removing parameter-ratio obstructions---remains open.

The present work aims to fill these gaps. Our main contributions are threefold:
\begin{enumerate}
\item \textbf{Cattaneo heat conduction.} We incorporate the Cattaneo heat conduction law, characterized by the relaxation time $\tau$ and heat flux $q$, which accurately captures the thermal delay inherent in real materials and provides a more realistic description of the thermal behavior of suspension bridges than the classical Fourier law.
\item \textbf{Switchable viscoelastic memory.} The adjustable parameter $\delta$ enables a unified analysis of both memoryless ($\delta = 0$) and memory-dependent ($\delta = 1$) scenarios. This fills the gap of ``switchable viscoelasticity'' in existing research and allows us to isolate the role of viscoelastic memory in the stability of the system.
\item \textbf{Comprehensive stability analysis and the regularizing role of memory.} The model integrates four core mechanisms---main cable vertical vibration, bridge deck transverse and rotational vibration based on Timoshenko beam theory, Cattaneo-type heat conduction, and thermo-mechanical coupling---while also including main cable internal damping and suspender stiffness. We establish well-posedness, derive sharp polynomial decay rates, identify conditions under which exponential stability fails, and prove that viscoelastic memory restores exponential stability regardless of the structural parameter configuration. This reveals a regularizing effect of viscoelastic memory that has not been previously observed in the context of Cattaneo thermoelasticity.
\end{enumerate}

\medskip
\noindent\textbf{Main results.}
The main results of this paper concerning the asymptotic stability of system \eqref{model1.1} are summarized as follows.
\begin{enumerate}
  \item When $\delta = 0$ (i.e., the system without viscoelastic memory), the solution to problem \eqref{model1.1} exhibits polynomial decay. More precisely, the decay rate depends on the relationship between the structural parameters of the bridge deck: if $\frac{\rho_1}{\kappa} = \frac{\rho_2}{b}$, the solution decays polynomially at the sharp rate $t^{-1/2}$; otherwise, when $\frac{\rho_1}{\kappa} \neq \frac{\rho_2}{b}$, the decay rate deteriorates to $t^{-1/4}$. Moreover, if $\chi_0 \neq 0$, or if $\chi_0 = 0$ and $\chi_1 = 0$, the solution to problem \eqref{model1.1} is not exponentially stable, where
  \[
  \chi_0 = \kappa\left(\dfrac{\rho_2}{b} -  \dfrac{\rho_1}{\kappa}\right)\left(\rho_3 - \dfrac{\rho_2}{b\tau}\right) + \dfrac{m^2 \rho_2}{b},
  \quad
  \chi_1 = -m^2 - \kappa\rho_3 + \dfrac{\kappa\rho_2}{\tau b}.
  \]
  The coefficient $\chi_0$ arises from the leading-order balance between the shear/bending wave-speed mismatch, the thermal inertia, and the Cattaneo relaxation time; it controls whether the resolvent admits a Weyl sequence at order $\lambda^2$. The coefficient $\chi_1$, on the other hand, is the leading-order coefficient of the determinant of the resolvent system and governs the growth rate of the solution in the degenerate regime.
  \item When $\delta = 1$ (i.e., the system with viscoelastic memory), the solution to problem \eqref{model1.1} achieves exponential decay. This indicates that the introduction of viscoelastic memory significantly enhances the stability of the system, leading to much faster energy dissipation compared with the memoryless case.
\end{enumerate}

These results reveal a crucial insight into the role of viscoelastic memory in regulating the stability of thermoelastic suspension bridge systems. First, the polynomial stability in the memoryless case ($\delta = 0$) is consistent with the findings of Miranda et al. \cite{MR4887867}, which confirms the rationality of our model in capturing the intrinsic dynamic characteristics of suspension bridges without viscoelastic effects. However, our result further refines the decay rate by identifying the critical parameter ratio $\frac{\rho_1}{\kappa} = \frac{\rho_2}{b}$ that governs the polynomial decay speed; this is attributed to the more accurate Cattaneo heat conduction law adopted in our model. Second, the exponential stability in the memory case ($\delta = 1$) demonstrates the dominant role of viscoelastic memory in promoting energy dissipation. The memory kernel $g(s)$ introduces additional damping mechanisms that continuously dissipate the system's energy accumulated from historical deformations, thereby overcoming the limited dissipation capacity of thermoelastic effects alone. More importantly, the exponential decay result does not require the parameter ratio condition, implying that viscoelastic memory can uniformly improve the stability of the system regardless of the initial structural parameter configuration. This is precisely the \emph{regularizing role} of viscoelastic memory: it eliminates the parameter-ratio obstruction present in the memoryless case and upgrades polynomial decay to exponential decay. From an engineering perspective, these results provide valuable guidance for the design and maintenance of suspension bridges: for bridges in environments where viscoelastic memory effects are significant (e.g., long-span bridges with polymer materials), the inherent viscoelasticity can be leveraged to enhance structural stability; for memoryless structures, the parameter ratios $\frac{\rho_1}{\kappa}$ and $\frac{\rho_2}{b}$ should be carefully optimized to ensure an acceptable decay rate.

The remainder of this paper is organized as follows. In Section \ref{sec2}, we focus on the case without the viscoelastic memory term, corresponding to $\delta = 0$. Specifically, Subsection \ref{sec2.1} is devoted to proving the well-posedness of the system under this condition. In Subsection \ref{sec2.2}, we establish the polynomial stability of the system, followed by Subsection \ref{sec2.3}, where we further demonstrate that the system fails to exhibit exponential stability under certain specified conditions. Subsequently, Section \ref{sec3} is devoted to the analysis of the case with the viscoelastic memory term, i.e., $\delta = 1$. In this section, Subsection \ref{sec3.1} verifies the well-posedness of the memory-dependent system, and Subsection \ref{sec3.2} rigorously proves that the system achieves exponential stability in this scenario.

\section{The Memoryless Case ($\delta = 0$): Well-posedness and Polynomial Stability} \label{sec2}
When $\delta = 0$, the suspension bridge model \eqref{model1.1} reduces to the following memoryless form:
\begin{align}\label{eq:without memory}
	\begin{cases}
		\rho u_{tt} - \alpha u_{xx} - \beta(\varphi - u) + \gamma u_t = 0 & \text{in } (0,L) \times \mathbb{R}^+, \\
		\rho_1 \varphi_{tt} - \kappa(\varphi_x + \psi)_x + \beta (\varphi - u) + m \theta_x = 0 & \text{in } (0,L) \times \mathbb{R}^+, \\
		\rho_2 \psi_{tt} - b \psi_{xx} + \kappa(\varphi_x + \psi) - m \theta = 0 & \text{in } (0,L) \times \mathbb{R}^+, \\
		\rho_3 \theta_t + q_x + m (\varphi_x + \psi)_t = 0 & \text{in } (0,L) \times \mathbb{R}^+, \\
		\tau q_t + \sigma q + \theta_x = 0 & \text{in } (0,L) \times \mathbb{R}^+.
	\end{cases}
\end{align}
This memoryless system is subject to the following initial and boundary conditions. For $t \in \mathbb{R}^+$, the boundary conditions are
\begin{align}\label{without memory:boundary condition}
	&u_x(0,t) = u_x(L,t) = \varphi_x(0,t) = \varphi_x(L,t)=0\notag\\
  &\psi(0,t) = \psi(L,t) = \theta(0,t) = \theta(L,t) = 0,
\end{align}
and for $x \in (0, L)$, the initial conditions are
\begin{equation}\label{without memory:initial condition}
	\begin{cases}
		u(x,0) = u_0(x), \quad u_t(x,0) = v_0(x), \\
		\varphi(x,0) = \varphi_0(x), \quad \varphi_t(x,0) = \Phi_0(x), \\
		\psi(x,0) = \psi_0(x), \quad \psi_t(x,0) = \Psi_0(x), \\
		\theta(x,0) = \theta_0(x), \quad q(x,0) = q_0(x).
	\end{cases}
\end{equation}

We consider the standard Lebesgue and Sobolev spaces over the interval $(0,L)$, denoted by
\[
L^2 = L^2(0,L), \quad H^1 = H^1(0,L), \quad H_0^1 = H_0^1(0,L).
\]
The standard inner product and norm on \(L^2\) are denoted by \((\cdot,\cdot)\) and \(\|\cdot\|\), respectively. We introduce the following subspaces:
\begin{align*}
	&L_*^2 := \Bigl\{ \omega \in L^2 : \int_0^L \omega(x)\,dx = 0 \Bigr\}, \\
	&H_*^1 := H^1 \cap L_*^2, \\
	&L_g^2(\mathbb{R}^+; H_0^1) := \Bigl\{ \omega : \sqrt{g}\,\omega \in L^2\bigl(\mathbb{R}^+; H_0^1\bigr) \Bigr\},
\end{align*}
where the last space is equivalently characterized by the finiteness of the integral
\[
\int_0^{\infty} g(s)\, \|\omega_x(s)\|^2 \, ds < \infty.
\]

By virtue of Poincar\'e's inequality, which asserts that for any \(\phi \in H_0^1\) or \(\phi \in H_*^1\),
\[
\|\phi\|^2 \le c_p \|\phi_x\|^2
\]
with a positive constant \(c_p > 0\), the following Hilbert space structures are induced:
\begin{itemize}
	\item The spaces \(H_0^1\) and \(H_*^1\) are Hilbert spaces when equipped with the inner product \((\cdot_x, \cdot_x)\) and the corresponding norm \(\|\cdot_x\|\).

	\item The subspace \(L_*^2\) is a Hilbert space under the standard \(L^2\)-inner product \((\cdot,\cdot)\) and its associated norm \(\|\cdot\|\).

	\item The space \(L_g^2(\mathbb{R}^+; H_0^1)\) is a Hilbert space with respect to the inner product
	\[
	(\cdot, \cdot)_{L_g^2(\mathbb{R}^+; H_0^1)} := \int_0^{\infty} g(s)\, (\cdot_x, \cdot_x) \, ds
	\]
	and the corresponding norm
	\[
	\|\cdot\|_{L_g^2(\mathbb{R}^+; H_0^1)} := \Bigl( \int_0^{\infty} g(s)\, \|\cdot_x\|^2 \, ds \Bigr)^{1/2}.
	\]
\end{itemize}

Let \(v = u_t\), \(\Phi = \varphi_t\), and \(\Psi = \psi_t\). To reformulate the problem as an abstract Cauchy problem, we introduce the phase space
\[
\mathcal{H}_1 := H_*^1 \times L_*^2 \times H_*^1 \times L_*^2 \times H_0^1 \times L^2 \times L^2 \times L^2,
\]
which is a Hilbert space endowed with the inner product
\begin{align*}
	(U_1, U_2)_{\mathcal{H}_1} := &\rho(v_1, v_2) + \alpha (u_{1x}, u_{2x}) + \rho_1(\Phi_1, \Phi_2) + \rho_2(\Psi_1, \Psi_2) \\
&+ b (\psi_{1x},\psi_{2x}) + \beta (\varphi_1 - u_1, \varphi_2 - u_2) \\
&+ \kappa(\varphi_{1x} + \psi_1, \varphi_{2x} + \psi_2) + \rho_3(\theta_1, \theta_2) + \tau(q_1, q_2),
\end{align*}
and the associated norm
\begin{align*}
\|U\|_{\mathcal{H}_1}^2 := &\rho\|v\|^2 + \alpha\|u_x\|^2 + \rho_1\|\Phi\|^2 + \rho_2\|\Psi\|^2 \\
&+ b\|\psi_x\|^2 + \beta\|\varphi - u\|^2 + \kappa\|\varphi_x + \psi\|^2 + \rho_3\|\theta\|^2 + \tau \|q\|^2,
\end{align*}
where \(U := (u, v, \varphi, \Phi, \psi, \Psi, \theta, q) \in \mathcal{H}_1\) and \(U_i := (u_i, v_i, \varphi_i, \Phi_i, \psi_i, \Psi_i, \theta_i, q_i) \in \mathcal{H}_1\) for \(i = 1, 2\).

The system \eqref{eq:without memory}--\eqref{without memory:initial condition} can be rewritten as the following abstract Cauchy problem:
\begin{equation}\label{model-abstract-1}
	\begin{cases}
		\dfrac{d}{dt}U = \mathcal{A}_1U, & t > 0, \\
		U(0) = U_0 := (u_0, v_0, \varphi_0, \Phi_0, \psi_0, \Psi_0, \theta_0, q_0),
	\end{cases}
\end{equation}
where \(\mathcal{A}_1: \mathcal{D}(\mathcal{A}_1) \subset \mathcal{H}_1 \to \mathcal{H}_1\) is a linear operator defined by
\begin{equation}\label{definition:A-1}
	\mathcal{A}_1 U := \begin{pmatrix}
		v \\
		\rho^{-1} \left[ \alpha u_{xx} + \beta (\varphi - u) - \gamma v \right] \\
		\Phi \\
		\rho_1^{-1} \left[ \kappa (\varphi_x + \psi)_x - \beta (\varphi - u) - m \theta_x \right] \\
		\Psi \\
		\rho_2^{-1} \left[ b \psi_{xx} - \kappa (\varphi_x + \psi) + m\theta \right] \\
		\rho_3^{-1} \left[ -q_x - m (\Phi_x + \Psi) \right] \\
		\tau^{-1}(-\sigma q - \theta_x)
	\end{pmatrix}
\end{equation}
with domain
\begin{align*}
	\mathcal{D}(\mathcal{A}_1) := \bigg\{ U \in \mathcal{H}_1 :
	v, \Phi \in H_*^1; \ u_x, \varphi_x, \theta, \Psi \in H_0^1; \
	\psi \in H^2; \ q \in H^1
	\bigg\}.
\end{align*}

The main results of this section are as follows.

\begin{theorem}[Well-posedness]\label{Well-Posedness-1}
	The operator $\mathcal{A}_1$ defined in \eqref{definition:A-1} is the infinitesimal generator of a $C_0$-contraction semigroup $\{e^{t\mathcal{A}_1}\}_{t \ge 0}$. Accordingly, if $U_0 \in \mathcal{H}_1$, then problem \eqref{model-abstract-1} has a unique mild solution $U(t) = e^{t\mathcal{A}_1}U_0$ belonging to $C(\mathbb{R}^+, \mathcal{H}_1)$; if $U_0 \in \mathcal{D}(\mathcal{A}_1)$, then $U(t) = e^{t\mathcal{A}_1}U_0$ is a classical solution to problem \eqref{model-abstract-1} and lies in $C(\mathbb{R}^+, \mathcal{D}(\mathcal{A}_1)) \cap C^1(\mathbb{R}^+, \mathcal{H}_1)$.
\end{theorem}

\begin{theorem}[Polynomial stability]\label{Polynomial stability}
	Let \(\{e^{\mathcal{A}_1 t}\}_{t \geq 0}\) be the \(C_0\)-semigroup generated by the operator \(\mathcal{A}_1: \mathcal{D}(\mathcal{A}_1) \subset \mathcal{H}_1 \to \mathcal{H}_1\). Then \(\{e^{\mathcal{A}_1 t}\}_{t \geq 0}\) is polynomially stable, and the sharp decay rate depends on the relationship between the structural parameters of the bridge deck:
	\begin{enumerate}[(i)]
		\item If \(\frac{\rho_1}{\kappa} = \frac{\rho_2}{b}\), there exists a positive constant \(C > 0\), independent of \(t\), such that
		\[
		\|e^{\mathcal{A}_1 t} U_0\|_{\mathcal{H}_1} \leq \frac{C}{t^{1/2}} \|U_0\|_{\mathcal{D}(\mathcal{A}_1)} \quad \text{for all } t > 0 \text{ and } U_0 \in \mathcal{D}(\mathcal{A}_1);
		\]
		\item If \(\frac{\rho_1}{\kappa} \neq \frac{\rho_2}{b}\), there exists a positive constant \(C > 0\), independent of \(t\), such that
		\[
		\|e^{\mathcal{A}_1 t} U_0\|_{\mathcal{H}_1} \leq \frac{C}{t^{1/4}} \|U_0\|_{\mathcal{D}(\mathcal{A}_1)} \quad \text{for all } t > 0 \text{ and } U_0 \in \mathcal{D}(\mathcal{A}_1),
		\]
	\end{enumerate}
	where $U_0 \in \mathcal{D}(\mathcal{A}_1)$ and $\|U_0\|_{\mathcal{D}(\mathcal{A}_1)} := \|\mathcal{A}_1U_0\|_{\mathcal{H}_1}$.
\end{theorem}

\begin{theorem}[Non-exponential stability]\label{Non-exponential Stability}
	Let the parameter-dependent coefficients $\chi_0$ and $\chi_1$ be defined as
	\[
	\chi_0 = \kappa\left(\dfrac{\rho_2}{b} - \dfrac{\rho_1}{\kappa}\right)\left(\rho_3 - \dfrac{\rho_2}{b\tau}\right) + \dfrac{m^2 \rho_2}{b}, \quad
	\chi_1 = -m^2 - \kappa\rho_3 + \dfrac{\kappa\rho_2}{\tau b}.
	\]
	If either $\chi_0 \neq 0$, or $\chi_0 = 0$ and $\chi_1 = 0$, then the $C_0$-semigroup $\{e^{\mathcal{A}_1 t}\}_{t \geq 0}$ generated by $\mathcal{A}_1$ is not exponentially stable.
\end{theorem}

\begin{remark}
	It should be noted that for the case where $\chi_0 = 0$ and $\chi_1 \neq 0$, the exponential stability of the semigroup $\{e^{\mathcal{A}_1 t}\}_{t \geq 0}$ remains an open problem and requires further investigation. In this regime, the leading-order terms in the resolvent expansion cancel at order $\lambda^2$, and the next-order correction involves a delicate balance between the Cattaneo relaxation time $\tau$ and the thermal inertia $\rho_3$, which prevents the construction of a simple Weyl sequence as in the proof of Theorem \ref{Non-exponential Stability}. A more refined spectral analysis, possibly involving higher-order asymptotic expansions of the resolvent, is needed to settle this case.
\end{remark}

\subsection{Well-Posedness}\label{sec2.1}
The proof of Theorem \ref{Well-Posedness-1} proceeds via an application of the Lumer--Phillips theorem \cite{MR710486}. Consequently, the demonstration hinges on verifying the following two conditions: (1) $\mathcal{A}_1$ is dissipative; (2) $\mathcal{A}_1$ is maximal. We now verify these two conditions sequentially in Lemmas \ref{dissipative-1} and \ref{maximal-1}.

\begin{lemma}\label{dissipative-1}
	The operator $\mathcal{A}_1$ defined in \eqref{definition:A-1} is dissipative.
\end{lemma}
\begin{proof}
	For any $U \in \mathcal{D}(\mathcal{A}_1)$, direct calculation and simplification via integration by parts yield
	\begin{align*}
		{\rm Re}(\mathcal{A}_1U,U)_{\mathcal{H}_1} &= {\rm Re}\bigg\{
		\rho\left(\rho^{-1} \left[ \alpha u_{xx} + \beta (\varphi - u) - \gamma v \right], v\right) + \alpha\left(v_x, u_x\right)  \\
	&\quad + \rho_1\left(\rho_1^{-1} \left[ \kappa (\varphi_x + \psi)_x - \beta (\varphi - u) - m \theta_x \right], \Phi\right) \\
		&\quad + \rho_2\left(\rho_2^{-1} \left[ b \psi_{xx} - \kappa (\varphi_x + \psi) + m\theta \right], \Psi\right) \\
	&\quad + b\left(\Psi_x,\psi_x\right) + \beta\left(\Phi - v, \varphi - u\right)  + \kappa\left(\Phi_x + \Psi, \varphi_x + \psi\right)  \\
	&\quad + \rho_3\left(\rho_3^{-1} \left[ -q_x - m (\Phi_x + \Psi) \right], \theta\right) + \tau\left(\tau^{-1}(-\sigma q - \theta_x),q\right) \bigg\} \\
		&= -\gamma\|v\|^2  - \sigma \|q\|^2\\
		&\le 0.
	\end{align*}
	Thus, $\mathcal{A}_1$ satisfies the dissipativity condition.
\end{proof}

\begin{lemma}\label{maximal-1}
	$0 \in \varrho(\mathcal{A}_1)$, where $\varrho(\mathcal{A}_1)$ denotes the resolvent set of $\mathcal{A}_1$. Consequently, $\mathcal{A}_1$ is maximal.
\end{lemma}
\begin{proof}
	First, we prove that for any $F = (f_1,f_2,f_3,f_4,f_5,f_6,f_7,f_8) \in \mathcal{H}_1$, there exists a unique
	\[
	U = (u,v,\varphi,\Phi,\psi,\Psi,\theta,q)\in \mathcal{D}(\mathcal{A}_1)
	\]
	such that
	\begin{align*}
		\mathcal{A}_1U = F.
	\end{align*}
	From the definition of $\mathcal{A}_1$ in \eqref{definition:A-1}, this equation can be rewritten as
	\begin{equation}\label{eq:resolvent-system-A-1}
		\begin{cases}
			v = f_1, \\
			\alpha u_{xx} + \beta (\varphi - u) - \gamma v = \rho f_2, \\
			\Phi = f_3, \\
			\kappa (\varphi_x + \psi)_x - \beta (\varphi - u) - m \theta_x = \rho_1 f_4, \\
			\Psi = f_5, \\
			b \psi_{xx} - \kappa (\varphi_x + \psi) + m\theta = \rho_2 f_6, \\
			-q_x - m (\Phi_x + \Psi) = \rho_3 f_7, \\
			-\sigma q - \theta_x = \tau f_8.
		\end{cases}
	\end{equation}
	From equations \eqref{eq:resolvent-system-A-1}$_{1,3,5}$, it follows that
	\begin{equation}\label{df1}
		v = f_1 \in H_*^1, \quad \Phi = f_3 \in H_*^1, \quad \Psi = f_5 \in H_0^1.
	\end{equation}
	Equation \eqref{eq:resolvent-system-A-1}$_8$ yields
	$q_x=-\frac{1}{\sigma}(\theta_{xx}+\tau f_{8x})$. Substituting this into equation \eqref{eq:resolvent-system-A-1}$_7$, we obtain
	\begin{align*}
		- \theta_{xx} = - \sigma m(\Phi_x + \Psi) - \sigma \rho_3 f_7 + \tau f_{8x}.
	\end{align*}
	Combining this with equations \eqref{eq:resolvent-system-A-1}$_{2,4,6}$, we derive
	\begin{align}\label{eq:resolvent-system-1-A-1}
		\begin{cases}
			-\alpha u_{xx} - \beta(\varphi - u) = -\rho f_2 - \gamma v=-\rho f_2 - \gamma f_1 := g_1 \in L_*^2, \\
			-\kappa(\varphi_x + \psi)_x + \beta(\varphi - u) + m\theta_x = -\rho_1 f_4 := g_2 \in L_*^2, \\
			- b \psi_{xx} + \kappa(\varphi_x + \psi) - m\theta = -\rho_2 f_6  := g_3 \in L^2, \\
			- \theta_{xx} = - \sigma m(\Phi_x + \Psi) - \sigma \rho_3 f_7 + \tau f_{8x} \\
\quad\quad\quad\quad\quad\quad\quad= - \sigma m(f_{3x} + f_5) - \sigma \rho_3 f_7 + \tau f_{8x}:=g_4 \in H^{-1}.
		\end{cases}
	\end{align}
	The weak formulation of equations \eqref{eq:resolvent-system-1-A-1} is as follows:
	\begin{align}\label{eq:resolvent-system-2-A-1}
		\begin{cases}
			\alpha(u_x, \tilde{u}_x) - \beta(\varphi - u, \tilde{u}) = (g_1, \tilde{u}), \\
			\kappa(\varphi_x + \psi, \dt \varphi_x) + \beta(\varphi - u, \dt \varphi) - m(\theta, \dt \varphi_x) = (g_2, \dt \varphi), \\
			b(\psi_x, \dt \psi_x) + \kappa(\varphi_x + \psi, \dt \psi) - m(\theta, \dt \psi) = (g_3, \dt \psi), \\
			\left(\theta_x, \frac{m^2 c_p}{\kappa } \dt \theta_x\right) = \left(g_4, \frac{m^2 c_p}{\kappa } \dt \theta\right),
		\end{cases}
	\end{align}
	for any $(\tilde{u}, \tilde{\varphi}, \tilde{\psi}, \tilde{\theta}) \in V := H_*^1 \times H_*^1 \times H_0^1 \times H_0^1$. Summing each equation in \eqref{eq:resolvent-system-2-A-1} yields the bilinear form $\mathcal{L}: V \times V \to \mathbb{R}$ and the linear form $\mathcal{F}:V \to \mathbb{R}$ satisfying
	\begin{equation}\label{LF}
		\mathcal{L}\left(
		\begin{pmatrix}u\\ \varphi \\ \psi \\ \theta\end{pmatrix},
		\begin{pmatrix}\tilde{u}\\ \tilde{\varphi} \\ \tilde{\psi} \\ \tilde{\theta}\end{pmatrix}
		\right) = \mathcal{F}\left(
		\begin{pmatrix}\tilde{u}\\ \tilde{\varphi} \\ \tilde{\psi} \\ \tilde{\theta}\end{pmatrix}
		\right),
	\end{equation}
	where $\mathcal{L}$ and $\mathcal{F}$ are defined as
\begin{align}\quad
\mathcal{L}\left(
			\begin{pmatrix}u\\ \varphi \\ \psi \\ \theta\end{pmatrix},
			\begin{pmatrix}\tilde{u}\\ \tilde{\varphi} \\ \tilde{\psi} \\ \tilde{\theta}\end{pmatrix}
			\right) &= \alpha (u_x, \tilde{u}_x) + \beta (\varphi - u, \tilde{\varphi} - \tilde{u}) + \kappa (\varphi_x + \psi, \tilde{\varphi}_x + \tilde{\psi}) \notag\\
			&\quad - m(\theta, \tilde{\varphi}_x+ \tilde{\psi}) + b (\psi_x, \tilde{\psi}_x)   + \frac{m^2 c_p}{\kappa} (\theta_x, \tilde{\theta}_x),
\end{align}
	\begin{align}\label{F}
		\mathcal{F}\left(
		\begin{pmatrix}\tilde{u}\\ \tilde{\varphi} \\ \tilde{\psi} \\ \tilde{\theta}\end{pmatrix}
		\right) &= \left(g_1, \tilde{u}\right) + \left(g_2, \tilde{\varphi}\right) + \left(g_3, \tilde{\psi}\right) + \left(g_4, \frac{m^2 c_p}{\kappa }\tilde{\theta}\right) \notag \\
		&=  - \gamma (f_1, \tilde{u}) - \rho (f_2, \tilde{u}) - \rho_1 (f_4, \tilde{\varphi}) - \rho_2 (f_6, \tilde{\psi})  \notag \\
		&\quad  - \frac{\sigma m^3 c_p }{\kappa } (f_{3x} + f_5, \tilde{\theta}) - \frac{\sigma \rho_3 m^2 c_p }{\kappa } (f_7, \tilde{\theta}) \notag\\
&\quad - \frac{\tau m^2 c_p}{\kappa}(f_8, \tilde{\theta}_x).
	\end{align}
	Clearly, $\mathcal{L}$ and $\mathcal{F}$ are bounded. We now prove that $\mathcal{L}$ is coercive. Indeed, for any $(u, \varphi, \psi, \theta) \in V$, by Young's inequality and Poincar\'e's inequality, we get
	\begin{align*}\label{L-1}
		\mathcal{L}\left(
		\begin{pmatrix}u\\ \varphi \\ \psi \\ \theta\end{pmatrix},
		\begin{pmatrix}u\\ \varphi \\ \psi \\ \theta\end{pmatrix}
		\right) &= \alpha \| u_x \|^2 + \beta \| \varphi - u \|^2 + \kappa \| \varphi_x + \psi \|^2 + b \| \psi_x \|^2 \\
&\quad + \frac{m^2 c_p}{\kappa} \| \theta_x \|^2 - m (\theta, \varphi_x + \psi)\\
		&\geq \alpha \| u_x \|^2 + \beta \| \varphi - u \|^2 + \kappa \| \varphi_x + \psi \|^2 + b \| \psi_x \|^2 \\
&\quad + \frac{m^2 c_p}{\kappa} \| \theta_x \|^2 - |m|\|\theta\|\|\varphi_x+\psi\|\\
		&\geq \alpha \| u_x \|^2 + \beta \| \varphi - u \|^2 + \frac{\kappa}{2} \| \varphi_x + \psi \|^2 \\
&\quad + b \| \psi_x \|^2 + \frac{m^2 c_p}{2\kappa} \| \theta_x \|^2 \\
		&\geq c \left\|
		\begin{pmatrix}u\\ \varphi \\ \psi \\ \theta\end{pmatrix}
		\right\|_V^2,
	\end{align*}
	where $c$ is a positive constant. Therefore, $\mathcal{L}$ is coercive. By the Lax--Milgram theorem, system \eqref{eq:resolvent-system-1-A-1} has a unique solution $(u, \varphi, \psi, \theta) \in V$.

	From equation \eqref{eq:resolvent-system-A-1}$_8$, it follows that
	\begin{align}
		q=-\frac{1}{\sigma}(\theta_x+\tau f_8) \in L^2.
	\end{align}
	From \eqref{eq:resolvent-system-A-1}$_7$, we also obtain $q_x=-m(\Phi_x+\Psi)-\rho_3 f_7 \in L^2$. Thus $q \in H^1$.

	Using standard elliptic regularity theory for \eqref{eq:resolvent-system-1-A-1}$_{1,2,3}$, we derive
	\begin{align}\label{jlll}
		u, \varphi, \psi \in H^2.
	\end{align}

	Setting $\dt \varphi = \dt \psi = \dt \theta = 0$ in \eqref{LF} yields
	\begin{align*}
		\alpha(u_x,\dt u_x) = (\beta(\varphi - u) + g_1, \dt u), \quad \forall \dt u \in H_*^1.
	\end{align*}
	Since $u\in H^2$, we can integrate by parts and rearrange the terms to obtain
	\begin{equation}\label{vb}
		(\alpha u_{xx}+\beta(\varphi - u) + g_1,\dt u)=\alpha \left(u_x(L)\dt u(L)-u_x(0)\dt u(0)\right), \quad \forall \dt u \in H_*^1.
	\end{equation}

	For $m$ an integer with $m\gg1$, we choose a nonnegative smooth function $\xi_m(x)$ defined on $\left[0,\frac{1}{m}\right]$ such that
	\[
	\xi_m(0)=1, \quad 0\le\xi_m(x)\le 1 \text{ for } x\in\left(0,\frac{1}{m}\right), \quad \xi_m\left(\frac{1}{m}\right)=0.
	\]
	Define
	\[
	\tilde{u}_m(x)=
	\begin{cases}
		\xi_m(x), & 0\le x\le\frac{1}{m},\\
		0, & \frac{1}{m}\le x\le L-\frac{1}{m},\\
		-\xi_m(L-x), & L-\frac{1}{m}\le x\le L.
	\end{cases}
	\]
	A direct calculation shows that $\tilde{u}_m(x)\in H_*^1$ and $\|\tilde{u}_m\|^2 \leq \frac{2}{m}$. Then, by \eqref{vb}, we have
	\begin{equation}\label{vb1}
		(\alpha u_{xx}+\beta(\varphi - u) + g_1,\tilde{u}_m)=\alpha \left(-u_x(L)-u_x(0)\right).
	\end{equation}
	By the Cauchy--Schwarz inequality, we derive
	\[
	\left| (\alpha u_{xx}+\beta(\varphi - u) + g_1,\tilde{u}_m)\right|\le\left(\frac{2}{m}\right)^{\frac{1}{2}}\|\alpha u_{xx}+\beta(\varphi - u) + g_1\|\to0
	\]
	as $m\to\infty$. Letting $m\to\infty$ in \eqref{vb1} gives
	\begin{equation}\label{vb2}
		u_x(L)+u_x(0)=0.
	\end{equation}
	Let
	\[
	\tilde{u}_m(x)=
	\begin{cases}
		\xi_m(x), & 0\le x\le \frac{1}{m},\\
		0, & \frac{1}{m}\le x\le\frac{L}{2}-\frac{1}{m},\\
		-\xi_m(\frac{L}{2}-x), & \frac{L}{2}-\frac{1}{m}\le x\le\frac{L}{2},\\
		-\xi_m(-\frac{L}{2}+x), & \frac{L}{2}\le x\le\frac{L}{2}+\frac{1}{m},\\
		0, & \frac{L}{2}+\frac{1}{m}\le x\le L-\frac{1}{m},\\
		\xi_m(L-x), & L-\frac{1}{m}\le x\le L.
	\end{cases}
	\]
	A direct calculation shows that $\tilde{u}_m(x)\in H_*^1$ and $\|\tilde{u}_m\|^2 \le\frac{4}{m}$. Then, by \eqref{vb}, we have
	\begin{equation}\label{gvb1}
		(\alpha u_{xx}+\beta(\varphi - u) + g_1,\tilde{u}_m)=\alpha \left(u_x(L)-u_x(0)\right).
	\end{equation}
	By the Cauchy--Schwarz inequality, we derive
	\[
	\left| (\alpha u_{xx}+\beta(\varphi - u) + g_1,\tilde{u}_m)\right|\le\frac{2}{m^{\frac{1}{2}}}\|\alpha u_{xx}+\beta(\varphi - u) + g_1\|\to0
	\]
	as $m\to\infty$. Letting $m\to\infty$ in \eqref{gvb1} gives
	\begin{equation}\label{gvb2}
		u_x(L)-u_x(0)=0.
	\end{equation}

	From \eqref{vb2} and \eqref{gvb2}, it follows that
	\[
	u_x(L)=u_x(0)=0.
	\]
	Substituting this back into \eqref{vb}, we obtain
	\[
	(\alpha u_{xx}+\beta(\varphi - u) + g_1,\tilde{u})=0, \quad \forall \tilde{u}\in H_*^1,
	\]
	which implies
	\[
	\alpha u_{xx}+\beta(\varphi - u) + g_1=0.
	\]
	Thus, we have
	\begin{equation}\label{jl6}
		u\in H^2\text{ with } u_x\in H_0^1 \text{ satisfies } \alpha u_{xx}+\beta(\varphi - u) + g_1=0.
	\end{equation}

	By similar reasoning,
	\begin{equation}\label{jl7}
		\varphi\in H^2\text{ with } \varphi_x\in H_0^1\text{ satisfies } -\kappa(\varphi_x + \psi)_x + \beta(\varphi - u) + m\theta_x = g_2.
	\end{equation}
	Consequently, $U \in \mathcal{D}(\mathcal{A}_1)$ is obtained. Clearly, the above proof shows that $\|U\|_{\mathcal{H}_1}=\|\mathcal{A}_1^{-1}F\|_{\mathcal{H}_1} \leq C\|F\|_{\mathcal{H}_1}$ for some positive constant $C$. Thus, $0 \in \varrho(\mathcal{A}_1)$.
\end{proof}

\subsection{Polynomial Stability}\label{sec2.2}
In this section, we establish the polynomial stability of the \(C_0\)-semigroup generated by the operator \(\mathcal{A}_1\), which corresponds to Theorem \ref{Polynomial stability}. Our analysis is conducted by distinguishing two critical scenarios based on the wave speeds associated with the bridge deck: the case of equal wave speeds and the case of distinct wave speeds.

The proof of Theorem \ref{Polynomial stability} relies on the following classical stability criterion for \(C_0\)-semigroups on Hilbert spaces, which characterizes exponential and polynomial stability in terms of the resolvent estimate of the generator.

\begin{theorem}\label{Stability Theorem}\cite{MR1681343,MR2606945}
	Let \(X\) be a Hilbert space, and let \(\{e^{\mathcal{A} t}\}_{t \geq 0}\) be a contraction \(C_0\)-semigroup on \(X\) with generator \(\mathcal{A}: \mathcal{D}(\mathcal{A}) \subset X \to X\). Suppose that the imaginary axis is contained in the resolvent set of \(\mathcal{A}\), i.e., \(i\mathbb{R} \subset \rho(\mathcal{A})\). Then the following assertions hold:
	\begin{enumerate}[(i)]
		\item The semigroup \(\{e^{\mathcal{A} t}\}_{t \geq 0}\) is exponentially stable, i.e., $\|e^{\mathcal{A} t}\|_{\mathcal{L}(X \to X)} \le Me^{-\omega t}$ for some constants $M,\omega>0$, if and only if
        \[
        \limsup_{|\lambda| \to \infty} \|(i\lambda I - \mathcal{A})^{-1}\|_{\mathcal{L}(X \to X)} < \infty;
        \]
		\item The semigroup \(\{e^{\mathcal{A} t}\}_{t \geq 0}\) is polynomially stable of order \(\frac{1}{l}\) (where \(l > 0\)), i.e., $\|e^{\mathcal{A} t}\phi\|_X\le Mt^{-\frac1l}\|\phi\|_{\mathcal{D}(\mathcal{A})}$ for $\phi\in \mathcal{D}(\mathcal{A})$ and some constant $M>0$, if and only if
        \[
        \limsup_{|\lambda| \to \infty} |\lambda|^{-l} \|(i\lambda I - \mathcal{A})^{-1}\|_{\mathcal{L}(X \to X)} < \infty.
        \]
	\end{enumerate}
\end{theorem}

For any $F = (f_1,f_2,f_3,f_4,f_5,f_6,f_7,f_8) \in \mathcal{H}_1$, consider the resolvent equation $(i\lambda I-\mathcal{A}_1)U=F$, where $U=(u,v,\varphi,\Phi,\psi,\Psi,\theta,q) \in \mathcal{D}(\mathcal{A}_1)$. From the definition of $\mathcal{A}_1$ in \eqref{definition:A-1}, it follows that
\begin{align}\label{Resolvent equations}
	\begin{cases}
		i\lambda u - v=f_1   \quad &\text{in } H_*^1,\\
		i\lambda \rho v -\alpha u_{xx}-\beta(\varphi - u)+ \gamma v =\rho f_2 \quad &\text{in } L_*^2,\\
		i\lambda\varphi - \Phi =f_3\quad &\text{in } H_*^1,\\
		i\lambda\rho_1 \Phi - \kappa(\varphi_{x}+\psi)_x + \beta (\varphi - u) +  m\theta_{x}=\rho_1f_4  \quad &\text{in } L_*^2,\\
		i\lambda \psi - \Psi =f_5 \quad &\text{in } H_0^1,\\
		i\lambda \rho_2 \Psi - b \psi_{xx} + \kappa(\varphi_{x} + \psi) - m\theta =\rho_2f_6 \quad &\text{in } L^2,\\
		i\lambda \rho_3 \theta + q_x + m(\Phi_{x} + \Psi) =\rho_3f_7 \quad &\text{in } L^2,\\
		i\lambda \tau q + \sigma q + \theta_x = \tau f_8 \quad &\text{in } L^2.
	\end{cases}
\end{align}

\begin{lemma}\label{lemma-iR}
	The resolvent set of the operator $\mathcal{A}_1$ contains the imaginary axis, i.e., $i\mathbb{R} \subset \rho(\mathcal{A}_1)$.
\end{lemma}
\begin{proof}
	By contradiction, suppose $i\mathbb{R} \not\subset \varrho(\mathcal{A}_1)$. Since $0 \in \varrho(\mathcal{A}_1)$, there exists a real number $\lambda \neq 0$ such that $i\lambda \not\in \varrho(\mathcal{A}_1)$, i.e., $i\lambda \in \sigma(\mathcal{A}_1)$. Since $\mathcal{A}_1$ is closed and $\mathcal{D}(\mathcal{A}_1)$ is compactly embedded in $\mathcal{H}_1$, the spectrum $\sigma(\mathcal{A}_1)$ contains only eigenvalues. Thus, $i\lambda$ is an eigenvalue and therefore corresponds to a nonzero eigenfunction $U=(u,v,\varphi,\Phi,\psi,\Psi,\theta,q)$. Hence, $F=(i\lambda I-\mathcal{A}_1)U = 0$.

	Taking the $\mathcal{H}_1$-inner product of $(i\lambda I-\mathcal{A}_1)U$ with $U$, from the proof of Lemma \ref{dissipative-1}, we obtain
	\begin{align*}
		0={\rm Re}(F,U)_{\mathcal{H}_1}={\rm Re}((i\lambda I-\mathcal{A}_1)U,U)_{\mathcal{H}_1}=-{\rm Re}(\mathcal{A}_1U,U)_{\mathcal{H}_1}=\gamma\|v\|^2+\sigma\|q\|^2.
	\end{align*}
	Thus, $v=q=0$. It is clear from equations \eqref{Resolvent equations} that $u=\varphi=\Phi=\psi=\Psi=\theta=0$. Therefore, $U=0$, contradicting $U \neq 0$. Consequently, $i\mathbb{R} \subset \rho(\mathcal{A}_1)$.
\end{proof}

\begin{lemma}\label{lemma-v,q}
	The following estimates hold:
	\begin{align}
		\|v\|^2 \leq \frac{1}{\gamma}\|U\|_{\mathcal{H}_1}\|F\|_{\mathcal{H}_1}, \label{estimate-v}\\
		\|q\|^2 \leq \frac{1}{\sigma}\|U\|_{\mathcal{H}_1}\|F\|_{\mathcal{H}_1}. \label{estimate-q}
	\end{align}
\end{lemma}
\begin{proof}
	Taking the $\mathcal{H}_1$-inner product of $(i\lambda I-\mathcal{A}_1)U=F$ with $U$ and computing the real part, we obtain
	\begin{align*}
		\|F\|_{\mathcal{H}_1}\|U\|_{\mathcal{H}_1}
		&\geq |\operatorname{Re}(F,U)_{\mathcal{H}_1}| = |\operatorname{Re}(i\lambda U-\mathcal{A}_1U,U)_{\mathcal{H}_1}| \\
		&= |\operatorname{Re}\{i\lambda(U,U)_{\mathcal{H}_1} - (\mathcal{A}_1U,U)_{\mathcal{H}_1}\}| = |-\operatorname{Re}(\mathcal{A}_1U,U)_{\mathcal{H}_1}|\\
		&= \gamma \|v\|^2 + \sigma \|q\|^2.
	\end{align*}
	Estimates \eqref{estimate-v} and \eqref{estimate-q} follow directly.
\end{proof}

\begin{lemma}\label{lemma-ux}
	There exists a constant $C>0$ such that
	\begin{align}\label{estimate-ux-1}
		\|u_x\|^2 \leq C(\|U\|_{\mathcal{H}_1}\|F\|_{\mathcal{H}_1} + \|v\|\|\varphi\|).
	\end{align}
	In particular, for any $\epsilon>0$, there exists a constant $C_{\epsilon}>0$, independent of $\lambda$, such that
	\begin{align}\label{estimate-ux-2}
		\|u_x\|^2 \leq \epsilon\|U\|_{\mathcal{H}_1}^2 + C_{\epsilon}\|F\|_{\mathcal{H}_1}^2
	\end{align}
	for $|\lambda| \geq 1$.
\end{lemma}
\begin{proof}
	Taking the $L^2$-inner product of equation \eqref{Resolvent equations}$_2$ with $u$ and integrating by parts, we obtain
	\begin{align}\label{eq2.1}
		i \lambda \rho (v, u) + \alpha \| u_x \|^2 - \beta (\varphi, u) + \beta \| u \|^2 + \gamma (v, u) = \rho (f_2, u).
	\end{align}
	By \eqref{Resolvent equations}$_1$, we get
	\begin{align}
		&i \lambda \rho (v, u) = - \rho(v,i\lambda u) = - \rho(v,v+f_1) = -\rho \|v\|^2 - \rho (v,f_1),\label{eq2.2}\\
		&\beta(\varphi,u) = \beta(\varphi,\frac{1}{i\lambda}(v+f_1))=\frac{i\beta}{\lambda}(\varphi,v)+\frac{i\beta}{\lambda}(\varphi,f_1).\label{eq2.3}
	\end{align}
	Substituting \eqref{eq2.2} and \eqref{eq2.3} into \eqref{eq2.1}, we obtain
	\begin{align*}
		\alpha \| u_x \|^2 + \beta \| u \|^2
		=& \rho \| v \|^2 + \rho (v, f_1) \\
&+ \frac{i\beta}{\lambda} (\varphi, v) +\frac{i\beta}{ \lambda} (\varphi, f_1) - \gamma (v, u) + \rho (f_2, u).
	\end{align*}
	Thus,
	\begin{align*}
		\alpha \|u_x\|^2 + \beta \|u\|^2
		\leq C\|U\|_{\mathcal{H}_1}\|F\|_{\mathcal{H}_1} + \beta \|\varphi\|\|v\| + \gamma \|v\|\|u\|,
	\end{align*}
	for $|\lambda| \geq 1$. By Young's inequality, we get
	\begin{align*}
		\|u_x\|^2 \leq C(\|U\|_{\mathcal{H}_1}\|F\|_{\mathcal{H}_1} + \|v\|\|\varphi\|).
	\end{align*}
	By estimate \eqref{estimate-v} and Young's inequality, we obtain
	\begin{align*}
		\|u_x\|^2\leq C\|U\|_{\mathcal{H}_1}\|F\|_{\mathcal{H}_1} + (\epsilon \|\varphi\|^2 + C_{\epsilon}\|v\|^2) \leq \epsilon\|U\|_{\mathcal{H}_1}^2 + C_{\epsilon}\|F\|_{\mathcal{H}_1}^2.
	\end{align*}
\end{proof}

\begin{lemma}\label{lemma-theta}
	There exists a constant $C>0$ such that
	\begin{align}\label{estimate-thete-1}
		\|\theta\|^2 \leq C(\|U\|_{\mathcal{H}_1}\|F\|_{\mathcal{H}_1} + \|q\|\|\Psi\| + \|q\|\|\Phi\|).
	\end{align}
	In particular, for any $\epsilon>0$, there exists a constant $C_{\epsilon}>0$, independent of $\lambda$, such that
	\begin{align}\label{estimate-theta-2}
		\|\theta\|^2 \leq \epsilon\|U\|_{\mathcal{H}_1}^2 + C_{\epsilon}\|F\|_{\mathcal{H}_1}^2
	\end{align}
	for $|\lambda| \geq 1$.
\end{lemma}
\begin{proof}
	For any $x \in (0,L)$, integrating equation \eqref{Resolvent equations}$_8$ over $(0, x)$ yields
	\begin{align}\label{eq2.4}
		i \lambda \tau \int_{0}^{x} q(y)  dy + \sigma \int_{0}^{x} q(y) dy + \theta(x) = \tau \int_{0}^{x} f_8(y)  dy.
	\end{align}
	Taking the $L^2$-inner product of equation \eqref{eq2.4} with $\theta$, we obtain
	\begin{align}\label{eq-theta}
		R_1 + \sigma \left( \int_{0}^{x} q(y)  dy, \theta \right) + \|\theta\|^2 = \tau \left( \int_{0}^{x} f_8(y)  dy, \theta \right),
	\end{align}
	where $R_1$ is defined as
	\begin{align*}
		R_1 = i \lambda \tau \left( \int_{0}^{x} q(y)  dy, \theta \right).
	\end{align*}
	From equation \eqref{Resolvent equations}$_7$, we obtain
	\begin{align}\label{eq2.5}
		i\lambda \theta = f_7 - \frac{1}{\rho_3} q_x - \frac{m}{\rho_3}(\Phi_x + \Psi).
	\end{align}
	Substituting equation \eqref{eq2.5} into $R_1$, we get
	\begin{align*}
		R_1 &= -\tau \left( \int_{0}^{x} q(y)  dy, i\lambda \theta \right) \\
		&= -\tau \left( \int_{0}^{x} q(y)  dy,  f_7 - \frac{1}{\rho_3} q_x - \frac{m}{\rho_3} \Phi_x - \frac{m}{\rho_3} \Psi \right) \\
		&= -\tau \left( \int_{0}^{x} q(y)  dy,  f_7 \right) + \frac{\tau}{\rho_3} \left( \int_{0}^{x} q(y)  dy,  q_x \right) \\
&\quad + \frac{\tau m}{\rho_3} \left( \int_{0}^{x} q(y)  dy, \Phi_x \right) + \frac{\tau m}{\rho_3} \left( \int_{0}^{x} q(y)  dy,  \Psi \right)  \\
		&= -\tau \left( \int_{0}^{x} q(y)  dy,  f_7 \right) + \frac{\tau m}{\rho_3} \left( \int_{0}^{x} q(y)  dy,  \Psi \right) \\
&\quad + R_2 - \frac{\tau}{\rho_3} \|q\|^2 - \frac{\tau m}{\rho_3}(q,\Phi),
	\end{align*}
	where $R_2$ is given by
	\begin{align*}
		R_2=\frac{\tau}{\rho_3}  \int_{0}^{L} q(y)  dy  (\overline{q}(L) + m  \overline{\Phi}(L)).
	\end{align*}
	Integrating equation \eqref{Resolvent equations}$_7$ over $(x, L)$, we get
	\begin{align}\label{eq2.6}
		i\lambda\rho_3 \int_x^L \theta(y)  dy& + q(L) - q(x) \notag\\
&+ m\Phi(L) - m\Phi(x) + m \int_x^L \Psi(y)  dy = \rho_3 \int_x^L f_7(y)  dy.
	\end{align}
	Multiplying equation \eqref{eq2.6} by $\int_{0}^{L}\overline{q}(x)dx$, we derive
	\begin{align*}
		\frac{\rho_3}{\tau}\overline{R}_2 =& \big(q(L) + m\Phi(L)\big) \int_0^L \overline{q}(x) dx \\
		=& \big(q(x) + m\Phi(x)\big) \int_0^L \overline{q}(x) dx - R_3 \\
&- m \int_x^L \Psi(y) dy \int_0^L \overline{q}(x) dx  + \rho_3 \int_x^L f_7(y) dy \int_0^L \overline{q}(x) dx,
	\end{align*}
	where $R_3=i\lambda\rho_3 \int_x^L \theta(y) dy \int_0^L \overline{q}(x) dx $. Equation \eqref{Resolvent equations}$_8$ yields
	\begin{align}\label{eq2.7}
		i\lambda q=f_8 - \frac{\sigma}{\tau}q - \frac{1}{\tau}\theta_x.
	\end{align}
	From equation \eqref{eq2.7}, it follows that
	\begin{align*}
		R_3 &= -\rho_3 \int_x^L \theta(y)  dy \int_0^L \overline{i\lambda q(x)}  dx \\
		&= -\rho_3 \int_x^L \theta(y)  dy \int_0^L \left( \overline{f_8} - \frac{\sigma}{\tau}\overline{q} - \frac{1}{\tau} \overline{\theta_x} \right) dx \\
		&= -\rho_3 \int_x^L \theta(y)  dy \int_0^L \overline{f_8}(x)  dx + \frac{\sigma \rho_3}{\tau} \int_x^L \theta(y)  dy \int_0^L \overline{q}(x)  dx.
	\end{align*}
	Substituting $R_3$ into $R_2$, then $R_2$ into $R_1$, and further $R_1$ into equation \eqref{eq-theta}, and integrating the resulting equation over $(0, L)$, estimate \eqref{estimate-thete-1} follows from H\"older's inequality and Young's inequality. By estimate \eqref{estimate-q} and Young's inequality, estimate \eqref{estimate-theta-2} is a direct consequence of \eqref{estimate-thete-1}.
\end{proof}

\begin{lemma}\label{lemma-varphix}
	There exists a constant $C>0$ such that
	\begin{align}\label{estimate-varphix-1}
		\|\varphi_x + \psi\|^2 \leq C\bigg(&\|U\|_{\mathcal{H}_1}\|F\|_{\mathcal{H}_1} + \frac{1}{\lambda}\|U\|_{\mathcal{H}_1}\|F\|_{\mathcal{H}_1} \notag\\
&+ \|q\|\|\Psi\| + \|q\|\|\Phi\| + \frac{1}{\lambda}\|q\|\|\varphi - u \|\bigg).
	\end{align}
	In particular, for any $\epsilon>0$, there exists a constant $C_{\epsilon}>0$, independent of $\lambda$, such that
	\begin{align}\label{estimate-varphix-2}
		\|\varphi_x + \psi\|^2 \leq \epsilon\|U\|_{\mathcal{H}_1}^2 + C_{\epsilon}\|F\|_{\mathcal{H}_1}^2
	\end{align}
	for $|\lambda| \geq 1$.
\end{lemma}
\begin{proof}
	From equations \eqref{Resolvent equations}$_{3,5}$, it follows that
	\begin{align} \label{eq2.8}
		\Phi_{x} = i \lambda \varphi_{x} - f_{3x}, \quad \Psi = i \lambda \psi - f_{5}.
	\end{align}
	Substituting \eqref{eq2.8} into \eqref{Resolvent equations}$_7$, we obtain
	\begin{align}\label{eq2.9}
		i \lambda \rho_{3} \theta + q_x + i m \lambda (\varphi_{x} + \psi) = m (f_{3x} + f_{5}) + \rho_{3} f_{7}.
	\end{align}
	Taking the $L^2$-inner product of \eqref{eq2.9} with $\varphi_x+\psi$ and integrating by parts, we derive
	\begin{align}\label{eq-varphix}
		i \lambda \rho_3 (\theta, \varphi_x\!  +\!  \psi) \! - \! R_4 \! +\!  i m \lambda \|\varphi_x \! +\!  \psi\|^2 = (m(f_{3x} \! +\!  f_5) \! +\! \rho_3 f_7, \varphi_x + \psi),
	\end{align}
	where $R_4=(q, (\varphi_x + \psi)_x)$. Equation \eqref{Resolvent equations}$_4$ yields
	\begin{align}\label{eq2.10}
		\kappa (\varphi_x + \psi)_x = i \lambda \rho_1 \Phi + \beta (\varphi - u) + m \theta_x - \rho_1 f_4.
	\end{align}
	Substituting \eqref{eq2.10} into $R_4$, we get
	\begin{align*}
		R_4 &= \frac{1}{\kappa}(q, i \lambda \rho_1 \Phi + \beta (\varphi - u) + m \theta_x - \rho_1 f_4) \\
		&= -\frac{i \lambda \rho_1}{\kappa}(q, \Phi) + \frac{\beta}{\kappa}(q, \varphi - u) + R_5 - \frac{\rho_1}{\kappa}(q, f_4),
	\end{align*}
	where $R_5=\frac{m}{\kappa}(q, \theta_x)$. From equation \eqref{Resolvent equations}$_8$, it follows that
	\begin{align}\label{eq2.11}
		\theta_x = \tau f_8 - i\lambda \tau q - \sigma q.
	\end{align}
	Substituting \eqref{eq2.11} into $R_5$, we obtain
	\begin{align*}
		R_5 = \frac{m}{\kappa} (q, \tau f_8 - i \lambda \tau q - \sigma q) = \frac{m \tau}{\kappa} (q, f_8) + \frac{i m \lambda \tau}{\kappa} \| q \|^2 - \frac{m \sigma}{\kappa} \| q \|^2.
	\end{align*}
	Substituting $R_5$ into $R_4$ and then $R_4$ into \eqref{eq-varphix}, and using H\"older's inequality, we derive
	\begin{align*}
		&\|\varphi_x + \psi\|^2 \\
&\leq C\left(\|\theta\|\|\varphi_x \! +\!  \psi\| \! + \! \frac{1}{\lambda}\|U\|_{\mathcal{H}_1}\|F\|_{\mathcal{H}_1} \! +\!  \|q\|\|\Phi\| \! +\!  \frac{1}{\lambda}\|q\|\|\varphi \! -\!  u\| \! +\!  \|q\|^2 \! +\!  \frac{1}{\lambda}\|q\|^2\right) \\
		&\leq C\left(\|U\|_{\mathcal{H}_1}\|F\|_{\mathcal{H}_1} \! +\!  \frac{1}{\lambda}\|U\|_{\mathcal{H}_1}\|F\|_{\mathcal{H}_1} \! +\!  \|\theta\|\|\varphi_x \! + \! \psi\| \! +\!  \|q\|\|\Phi\| \! +\!  \frac{1}{\lambda}\|q\|\|\varphi \! -\!  u\|\right).
	\end{align*}
	Using estimate \eqref{estimate-thete-1} and Young's inequality, we obtain
	\begin{align*}
		&\|\varphi_x + \psi\|^2 \\
&\leq C\left(\|U\|_{\mathcal{H}_1}\|F\|_{\mathcal{H}_1} \! +\!  \frac{1}{\lambda}\|U\|_{\mathcal{H}_1}\|F\|_{\mathcal{H}_1} \! +\!  \|\theta\|^2 \! +\!  \|q\|\|\Phi\| \! +\!  \frac{1}{\lambda}\|q\|\|\varphi \! -\!  u\|\right)\\
		&\leq C\left(\|U\|_{\mathcal{H}_1}\|F\|_{\mathcal{H}_1} \! + \! \frac{1}{\lambda}\|U\|_{\mathcal{H}_1}\|F\|_{\mathcal{H}_1} \! +\!  \|q\|\|\Psi\|\!  +\!  \|q\|\|\Phi\| \! +\!  \frac{1}{\lambda}\|q\|\|\varphi \! -\!  u\|\right).
	\end{align*}
	The proof of \eqref{estimate-varphix-1} is complete. Using Young's inequality and estimate \eqref{estimate-q}, estimate \eqref{estimate-varphix-2} follows from \eqref{estimate-varphix-1} for $|\lambda| \geq 1$.
\end{proof}

\begin{lemma}\label{lemma-varphi-u}
	There exists a constant $C>0$ such that
	\begin{align}\label{estimate-varphi-u-1}
		\|\varphi - u\|^2 \leq C(\|U\|_{\mathcal{H}_1}\|F\|_{\mathcal{H}_1}  + \|v\|\|\Phi\| + \|v\|\|\varphi\| + \|u_x\|\|\varphi_x\|).
	\end{align}
	In particular, for any $\epsilon>0$, there exists a constant $C_{\epsilon}>0$, independent of $\lambda$, such that
	\begin{align}\label{estimate-varphi-u-2}
		\|\varphi - u\|^2 \leq \epsilon\|U\|_{\mathcal{H}_1}^2 + C_{\epsilon}\|F\|_{\mathcal{H}_1}^2
	\end{align}
	for $|\lambda| \geq 1$.
\end{lemma}
\begin{proof}
	Taking the $L^2$-inner product of equation \eqref{Resolvent equations}$_2$ with $\varphi-u$ and integrating by parts, we obtain
	\begin{align}\label{eq-varphi-u}
		R_6 + \alpha (u_x, \varphi_x - u_x) - \beta \|\varphi - u\|^2 + \gamma (v, \varphi - u) = \rho (f_2, \varphi - u),
	\end{align}
	where $R_6$ is defined as
	\begin{align*}
		R_6=i \lambda \rho (v, \varphi - u).
	\end{align*}
	Using equations \eqref{Resolvent equations}$_{1,3}$, we obtain
	\begin{align*}
		R_6 &= -\rho (v, i \lambda \varphi - i \lambda u)\\
		&= -\rho (v, \Phi + f_3 - v - f_1)
		= -\rho (v, \Phi) + \rho \|v\|^2 - \rho (v, f_3 - f_1).
	\end{align*}
	Substituting $R_6$ into equation \eqref{eq-varphi-u} and using H\"older's inequality, we derive
	\begin{align*}
		\|\varphi - u\|^2 \leq C(\|v\|\|\Phi\| + \|U\|_{\mathcal{H}_1}\|F\|_{\mathcal{H}_1} + \|u_x\|\|\varphi_x\| + \|u_x\|^2 + \|v\|\|\varphi - u\|).
	\end{align*}
	Using estimates \eqref{estimate-v} and \eqref{estimate-ux-1} and Young's inequality, we obtain \eqref{estimate-varphi-u-1} immediately. Using estimates \eqref{estimate-v} and \eqref{estimate-ux-2}, \eqref{estimate-varphi-u-2} follows from \eqref{estimate-varphi-u-1}.
\end{proof}

\begin{lemma}\label{lemma-Phi}
	There exists a constant $C>0$ such that
	\begin{align}\label{estimate-Phi-1}
		\|\Phi\|^2 \leq C(\|U\|_{\mathcal{H}_1}\|F\|_{\mathcal{H}_1}  + \|\varphi_x+\psi\|\|\varphi_x\| + \|\varphi-u\|\|\varphi\| + \|\theta\|\|\varphi_x\|).
	\end{align}
	In particular, for any $\epsilon>0$, there exists a constant $C_{\epsilon}>0$, independent of $\lambda$, such that
	\begin{align}\label{estimate-Phi-2}
		\|\Phi\|^2 \leq \epsilon\|U\|_{\mathcal{H}_1}^2 + C_{\epsilon}\|F\|_{\mathcal{H}_1}^2
	\end{align}
	for $|\lambda| \geq 1$.
\end{lemma}
\begin{proof}
	Taking the $L^2$-inner product of equation \eqref{Resolvent equations}$_4$ with $\varphi$ and integrating by parts, we obtain
	\begin{align}\label{eq-Phi}
		R_7 + \kappa (\varphi_x + \psi, \varphi_x) + \beta (\varphi - u, \varphi) - m (\theta, \varphi_x) = \rho_1 (f_4, \varphi),
	\end{align}
	where $R_7 = i \lambda \rho_1 (\Phi, \varphi)$. From equation \eqref{Resolvent equations}$_3$, it follows that
	\begin{align}
		R_7 = -\rho_1 (\Phi, i \lambda \varphi) = -\rho_1 (\Phi, \Phi + f_3) = -\rho_1 \|\Phi\|^2 - \rho_1 (\Phi, f_3).
	\end{align}
	Substituting $R_7$ into equation \eqref{eq-Phi} and using H\"older's inequality, we obtain \eqref{estimate-Phi-1} immediately. Using estimates \eqref{estimate-varphix-2}, \eqref{estimate-varphi-u-2} and \eqref{estimate-theta-2}, \eqref{estimate-Phi-2} follows from \eqref{estimate-Phi-1}.
\end{proof}

\begin{lemma}\label{lemma-psix}
	For $|\lambda| \geq \max\{1, \sigma\}$, the following inequalities hold:
	\begin{enumerate}[(i)]
		\item if $\frac{\rho_1}{\kappa}=\frac{\rho_2}{b}$, for any $\epsilon>0$, there exists a constant $C_{\epsilon}>0$, independent of $\lambda$, such that
		\begin{align}
			\|\psi_x\|^2 \leq \epsilon\|U\|_{\mathcal{H}_1}^2 + C_{\epsilon} \lambda ^4 \|F\|_{\mathcal{H}_1}^2;
		\end{align}
		\item if $\frac{\rho_1}{\kappa} \neq \frac{\rho_2}{b}$, for any $\epsilon>0$, there exists a constant $C_{\epsilon}>0$, independent of $\lambda$, such that
		\begin{align}
			\|\psi_x\|^2 \leq \epsilon\|U\|_{\mathcal{H}_1}^2 + C_{\epsilon}\lambda ^8\|F\|_{\mathcal{H}_1}^2.
		\end{align}
	\end{enumerate}
\end{lemma}
\begin{proof}
	From equations \eqref{Resolvent equations}$_{3,5}$, it follows that
	\begin{align}\label{eq2.12}
		i \lambda \varphi_x - \Phi_x = f_{3x}, \quad i \lambda \psi - \Psi = f_5.
	\end{align}
	Adding the two equations in \eqref{eq2.12}, we obtain
	\begin{align}\label{eq2.13}
		i \lambda (\varphi_x + \psi) - (\Phi_x + \Psi) = f_{3 x} + f_5.
	\end{align}
	Taking the $L^2$-inner product of equation \eqref{eq2.13} with $\Psi$ and rearranging, we derive
	\begin{align}\label{eq2.14}
		-i \lambda (\varphi_x + \psi, \Psi) - R_8 + R_9 = -(f_{3 x}, \Psi) - (f_5, \Psi),
	\end{align}
	where $R_8$ and $R_9$ are defined as $R_8 = (\Phi, \Psi_x)$ and $R_9 = (\Psi, \Psi)$. By equation \eqref{Resolvent equations}$_5$, substituting $\Psi_x$ and $\Psi$ into $R_8$ and $R_9$, we obtain
	\begin{align*}
		R_8 &= (\Phi, i \lambda \psi_x - f_{5 x}) = -i \lambda (\Phi, \psi_x) - (\Phi, f_{5 x}), \\
		R_9 &= (\Psi, i \lambda \psi - f_5) = -i \lambda (\Psi, \psi) - (\Psi, f_5).
	\end{align*}
	Substituting $R_8$ and $R_9$ into equation \eqref{eq2.14}, we get
	\begin{align}\label{eq2.15}
		-i \lambda (\varphi_x + \psi, \Psi) &+ i \lambda (\Phi, \psi_x)   -i \lambda (\Psi, \psi) \notag\\
& = -(f_{3 x}, \Psi) - (f_5, \Psi) - (\Phi, f_{5 x}) + (\Psi, f_5).
	\end{align}
	Taking the $L^2$-inner product of equation \eqref{Resolvent equations}$_4$ with $\frac{1}{\rho_1}\psi_x$ and the $L^2$-inner product of equation \eqref{Resolvent equations}$_6$ with $\frac{1}{\rho_2}\psi$, and rearranging, we derive
	\begin{align}
		i\lambda(\Phi,\psi_{x}) &= -\frac{\kappa}{\rho_{1}}(\varphi_{x} \! +\!  \psi, \psi_{xx}) \! -\!  \frac{\beta}{\rho_{1}}(\varphi\!  -\!  u, \psi_{x}) \! -\!  \frac{m}{\rho_{1}}(\theta_{x}, \psi_{x}) \! +\!  (f_{4}, \psi_{x}), \label{a1}\\
		i\lambda(\Psi,\psi) &= -\frac{b}{\rho_{2}}\|\psi_{x}\|^{2} - \frac{\kappa}{\rho_{2}}(\varphi_{x} + \psi, \psi) + \frac{m}{\rho_{2}}(\theta, \psi) + (f_{6}, \psi). \label{a2}
	\end{align}
	Substituting \eqref{a1} and \eqref{a2} into equation \eqref{eq2.15}, we obtain
	\begin{align}\label{l1}
		&-i \lambda (\varphi_x + \psi, \Psi) -\frac{\kappa}{\rho_{1}}(\varphi_{x} + \psi, \psi_{xx}) - \frac{\beta}{\rho_{1}}(\varphi - u, \psi_{x}) - \frac{m}{\rho_{1}}(\theta_{x}, \psi_{x})\notag\\& +\frac{b}{\rho_{2}}\|\psi_{x}\|^{2} + \frac{\kappa}{\rho_{2}}(\varphi_{x} + \psi, \psi) - \frac{m}{\rho_{2}}(\theta, \psi)\notag\\
		=& -(f_{3 x}, \Psi) - (f_5, \Psi) - (\Phi, f_{5 x}) + (\Psi, f_5) - (f_{4}, \psi_{x}) + (f_{6}, \psi).
	\end{align}
	Equation \eqref{Resolvent equations}$_6$ gives
	\begin{align*}
		-\frac{\kappa}{b \rho_1} ( \varphi_x + \psi, i \lambda \rho_2 \Psi - b \psi_{xx} + \kappa (\varphi_x + \psi) - m \theta ) = -\frac{\kappa}{b \rho_1} ( \varphi_x + \psi, \rho_2 f_6).
	\end{align*}
	That is,
	\begin{align}\label{l2}
		i \lambda \frac{\kappa \rho_2}{b \rho_1} ( \varphi_x + \psi, \Psi ) &+ \frac{\kappa}{\rho_1} ( \varphi_x + \psi, \psi_{xx}) - \frac{\kappa ^2}{b \rho_1} \| \varphi_x + \psi \|^2\notag\\
&+ \frac{\kappa m}{b \rho_1} ( \varphi_x + \psi, \theta ) = -\frac{\kappa \rho_2}{b \rho_1} ( \varphi_x + \psi, f_6).
	\end{align}
	Adding \eqref{l1} and \eqref{l2}, we have
	\begin{align}\label{l3}
		&i \lambda (\frac{\kappa \rho_2}{b \rho_1} - 1) (\varphi_x + \psi, \Psi)  - \frac{\beta}{\rho_{1}}(\varphi - u, \psi_{x}) - \frac{m}{\rho_{1}}(\theta_{x}, \psi_{x})\notag\\ &+\frac{b}{\rho_{2}}\|\psi_{x}\|^{2} + \frac{\kappa}{\rho_{2}}(\varphi_{x} + \psi, \psi) - \frac{m}{\rho_{2}}(\theta, \psi) - \frac{\kappa ^2}{b \rho_1} \| \varphi_x + \psi \|^2+ \frac{\kappa m}{b \rho_1} ( \varphi_x + \psi, \theta ) \notag\\
		 = &-(f_{3 x}, \Psi) - (f_5, \Psi) - (\Phi, f_{5 x}) + (\Psi, f_5) \notag\\
&- (f_{4}, \psi_{x}) + (f_{6}, \psi) -\frac{\kappa \rho_2}{b \rho_1} ( \varphi_x + \psi, f_6).
	\end{align}
	\paragraph{Case 1.} If $\frac{\rho_1}{\kappa}=\frac{\rho_2}{b}$, i.e., $\frac{\kappa \rho_2}{b \rho_1} = 1$. From equation \eqref{l3}, it follows that
	\begin{align*}
		\|\psi_x\|^2 \leq C \left(\|\varphi-u\|^2 + \|\theta_x\|^2 + \|\varphi_x+\psi\|^2 + \|U\|_{\mathcal{H}_1} \|F\|_{\mathcal{H}_1}  \right).
	\end{align*}
	Now we estimate $\|\theta_x\|^2$. By equation \eqref{Resolvent equations}$_8$, we derive
	\begin{align}\label{estimate-thetax}
		\|\theta_x\|^2 &= \|\tau f_8 - (i \lambda \tau + \sigma) q\|^2 \notag\\
&\leq 2 \tau^2 \|f_8\|^2 + C \lambda ^2 \|q\|^2 \leq 2 \tau^2 \|F\|_{\mathcal{H}_1}^2 + C \lambda^2 \|U\|_{\mathcal{H}_1} \|F\|_{\mathcal{H}_1} \notag\\
		&\leq 2 \tau^2 \|F\|_{\mathcal{H}_1}^2 + C \lambda ^2 \left( \frac{\epsilon}{\lambda ^2} \|U\|_{\mathcal{H}_1}^2 + C_{\epsilon} \lambda^2 \|F\|_{\mathcal{H}_1}^2 \right) \notag\\
&\leq \epsilon \|U\|_{\mathcal{H}_1}^2 + C_{\epsilon} \lambda^4 \|F\|_{\mathcal{H}_1}^2,
	\end{align}
	for $|\lambda| \geq \sigma$. From estimates \eqref{estimate-thetax}, \eqref{estimate-varphi-u-2} and \eqref{estimate-varphix-2}, it follows that
	\begin{align}
		\|\psi_x\|^2 \leq \epsilon \|U\|_{\mathcal{H}_1}^2 + C_{\epsilon} \lambda^4 \|F\|_{\mathcal{H}_1}^2,
	\end{align}
	for $|\lambda| \geq \max\{1, \sigma\}$.

	\paragraph{Case 2.} If $\frac{\rho_1}{\kappa}\neq\frac{\rho_2}{b}$, i.e., $\frac{\kappa \rho_2}{b \rho_1} \neq 1$. From equation \eqref{l3}, it follows that
	\begin{align*}
		\|\psi_x\|^2 \leq C \left(\lambda\|\varphi_x + \psi\|\|\Psi\| + \|\varphi-u\|^2 + \|\theta_x\|^2 + \|\varphi_x+\psi\|^2 + \|U\|_{\mathcal{H}_1} \|F\|_{\mathcal{H}_1}  \right).
	\end{align*}
	By Young's inequality, estimate \eqref{estimate-varphix-1} yields
	\begin{align*}
		\|\varphi_x + \psi\|^2 \leq C\left(\epsilon \|U\|_{\mathcal{H}_1}^2 + \frac{1}{\epsilon}\|U\|_{\mathcal{H}_1}\|F\|_{\mathcal{H}_1}\right).
	\end{align*}
	Consequently,
	\begin{align}\label{estimate1}
		\lambda\|\varphi_x + \psi\|\|\Psi\| &\leq \frac{\epsilon}{2}\|\Psi\|^2 + \frac{\lambda ^2}{2\epsilon} \|\varphi_x + \psi\|^2\notag\\
		&\leq \epsilon \|U\|_{\mathcal{H}_1}^2 + \frac{C \lambda ^2}{2\epsilon}\left(\frac{\epsilon^2}{\lambda^2}\|U\|_{\mathcal{H}_1}^2 + \frac{\lambda^2}{\epsilon^2}\|U\|_{\mathcal{H}_1}\|F\|_{\mathcal{H}_1}\right) \notag\\
		&\leq \epsilon \|U\|_{\mathcal{H}_1}^2 + \frac{C \lambda ^4}{\epsilon^3}\|U\|_{\mathcal{H}_1}\|F\|_{\mathcal{H}_1}\notag\\
		&\leq \epsilon \|U\|_{\mathcal{H}_1}^2 + C_{\epsilon}\lambda^8 \|F\|_{\mathcal{H}_1}^2.
	\end{align}
	From estimates \eqref{estimate1}, \eqref{estimate-thetax}, \eqref{estimate-varphi-u-2} and \eqref{estimate-varphix-2}, it follows that
	\begin{align*}
		\|\psi_x\|^2 \leq \epsilon \|U\|_{\mathcal{H}_1}^2 + C_{\epsilon} \lambda^8 \|F\|_{\mathcal{H}_1}^2,
	\end{align*}
	for $|\lambda| \geq \max\{1, \sigma\}$.
\end{proof}

\begin{lemma}\label{lemma-Psi}
	For $|\lambda| \geq \max\{1, \sigma\}$, the following inequalities hold:
	\begin{enumerate}[(i)]
		\item if $\frac{\rho_1}{\kappa}=\frac{\rho_2}{b}$, for any $\epsilon>0$, there exists a constant $C_{\epsilon}>0$, independent of $\lambda$, such that
		\begin{align}
			\|\Psi\|^2 \leq \epsilon\|U\|_{\mathcal{H}_1}^2 + C_{\epsilon} \lambda ^4 \|F\|_{\mathcal{H}_1}^2;
		\end{align}
		\item if $\frac{\rho_1}{\kappa} \neq \frac{\rho_2}{b}$, for any $\epsilon>0$, there exists a constant $C_{\epsilon}>0$, independent of $\lambda$, such that
		\begin{align}
			\|\Psi\|^2 \leq \epsilon\|U\|_{\mathcal{H}_1}^2 + C_{\epsilon}\lambda ^8\|F\|_{\mathcal{H}_1}^2.
		\end{align}
	\end{enumerate}
\end{lemma}
\begin{proof}
	Taking the $L^2$-inner product of equation \eqref{Resolvent equations}$_6$ with $\psi$ and using equation \eqref{Resolvent equations}$_5$, we obtain
	\begin{align*}
		-\rho_2\|\Psi\|^2 - \rho_2(\Psi,f_5) + b\|\psi_x\|^2 + \kappa(\varphi_x+\psi,\psi) - m(\theta,\psi) = \rho_2(f_6,\psi).
	\end{align*}
	Therefore, by H\"older's inequality,
	\begin{align*}
		\|\Psi\|^2 \leq C(\|\psi_x\|^2+\|\varphi_x+\psi\|^2+\|\theta\|^2+\|U\|_{\mathcal{H}_1}\|F\|_{\mathcal{H}_1}).
	\end{align*}
	Consequently, Lemma \ref{lemma-Psi} follows directly from Lemmas \ref{lemma-theta}, \ref{lemma-varphix} and \ref{lemma-psix}.
\end{proof}

\begin{proof}[Proof of Theorem \ref{Polynomial stability}]
	From Lemmas \ref{lemma-v,q}--\ref{lemma-Psi}, for $|\lambda| \geq \max\{1, \sigma\}$, the following estimates hold:
	\paragraph{Case 1.} If $\frac{\rho_1}{\kappa}=\frac{\rho_2}{b}$, there exists a constant $C> 0$, independent of $\lambda$, such that $\|U\|_{\mathcal{H}_1} \leq C \lambda^2\|F\|_{\mathcal{H}_1}$.
	\paragraph{Case 2.} If $\frac{\rho_1}{\kappa} \neq \frac{\rho_2}{b}$, there exists a constant $C> 0$, independent of $\lambda$, such that $\|U\|_{\mathcal{H}_1} \leq C \lambda^4\|F\|_{\mathcal{H}_1}$.

	Combined with Lemma \ref{lemma-iR}, by Theorem \ref{Stability Theorem}, we conclude that Theorem \ref{Polynomial stability} holds.
\end{proof}

\subsection{Failure of Exponential Stability}\label{sec2.3}
In this section, we demonstrate that the system fails to be exponentially stable under certain coefficient conditions and give the proof of Theorem \ref{Non-exponential Stability}.

\begin{proof}[Proof of Theorem \ref{Non-exponential Stability}]
	We employ the strategy of showing that the operators $(i\lambda I-\mathcal{A}_1)^{-1}$ are not uniformly bounded. Specifically, we will show that there exists a sequence of real numbers $\lambda_n$ with $|\lambda_n| \to \infty$, and a sequence of functions $F_n \in \mathcal{H}_1$ such that $\{F_n\}$ is bounded in $\mathcal{H}_1$, satisfying $(i\lambda_n I-\mathcal{A}_1)U_n = F_n$, where $U_n \in \mathcal{D}(\mathcal{A}_1)$ and $\|U_n\|_{\mathcal{H}_1} \to \infty$, $n \in \mathbb{N}$.

	The choice of $F_n$ is motivated by the need to excite the $\psi$-component at a frequency $\lambda_n = n\pi/(\omega L)$ that is resonant with the wave speed $\omega = \sqrt{\rho_2/b}$, thereby preventing the resolvent from being uniformly bounded. We define $F_n=(0,0,0,0,0,\rho_2 ^{-1}\sin\omega \lambda_n x,0,0) \in \mathcal{H}_1$, $\lambda_n = \frac{n\pi}{\omega L}$, where $\omega=\sqrt{\frac{\rho_2}{b}}$, $n \in \mathbb{N}$. For notational simplicity, we omit the subscript $n$ in what follows. From the definition of $\mathcal{A}_1$, $(i\lambda_n I-\mathcal{A}_1)U_n = F_n$ is equivalent to
	\begin{align}\label{j1}
		\begin{cases}
			i\lambda u - v=0,\\
			i\lambda \rho v -\alpha u_{xx}-\beta(\varphi - u)+ \gamma v =0,\\
			i\lambda\varphi - \Phi =0,\\
			i\lambda\rho_1 \Phi - \kappa(\varphi_{x}+\psi)_x + \beta (\varphi - u) +  m\theta_{x}=0  ,\\
			i\lambda \psi - \Psi =0 ,\\
			i\lambda \rho_2 \Psi - b \psi_{xx} + \kappa(\varphi_{x} + \psi) - m\theta =\sin\omega\lambda x ,\\
			i\lambda \rho_3 \theta + q_x + m(\Phi_{x} + \Psi) =0 ,\\
			i\lambda \tau q + \sigma q + \theta_x = 0 .
		\end{cases}
	\end{align}
	Equations \eqref{j1}$_{1,3,5}$ yield
	\begin{align}\label{j2}
		 v = i \lambda u, \quad \Phi = i \lambda \varphi, \quad \Psi = i \lambda \psi .
	\end{align}
	Substituting \eqref{j2} into equations \eqref{j1}$_{2,4,6,7,8}$, we obtain
	\begin{align}\label{j3}
		\begin{cases}
			& (i \lambda \gamma - \lambda^2 \rho + \beta) u - \alpha u_{xx} - \beta \varphi = 0, \\
			&(\beta - \lambda^2 \rho_1) \varphi - \kappa \varphi_{xx} - \kappa \psi_x - \beta u + m \theta_x = 0, \\
			&(\kappa - \lambda^2 \rho_2) \psi - b \psi_{xx} + \kappa \varphi_x - m \theta = \sin \omega \lambda x,\\
			& i \lambda \rho_3 \theta + q_x + i\lambda m \varphi_x + i\lambda m\psi = 0, \\
			& i\lambda \tau q + \sigma q + \theta_x = 0.
		\end{cases}
	\end{align}
	Given that $U \in \mathcal{D}(\mathcal{A}_1)$, we assume the solution of equations \eqref{j3} takes the form
	\begin{align}\label{j4}
		&u=A\cos\omega\lambda x, \quad \varphi=B\cos\omega\lambda x, \quad \psi=C\sin\omega\lambda x, \notag\\
 &\theta=D\sin\omega\lambda x, \quad q=E\cos\omega\lambda x.
	\end{align}
	Substituting \eqref{j4} into equations \eqref{j3}, we derive
	\begin{align}\label{j5}
		\begin{cases}
			(i \lambda \gamma - \rho \lambda^2 + \beta + \alpha \omega^2 \lambda^2) A - \beta B = 0, \\
			-\beta A + (\beta - \rho_1 \lambda^2 + \kappa \omega^2 \lambda^2) B - \kappa \omega \lambda C + m \omega \lambda D = 0, \\
			-\kappa \omega \lambda B + (\kappa - \rho_2 \lambda^2 + b \omega^2 \lambda^2) C - m D = 1, \\
			-im \omega \lambda B + im C + i \rho_3 D - \omega E = 0, \\
			\omega \lambda D + (i \lambda \tau + \sigma) E = 0.
		\end{cases}
	\end{align}
	From equations \eqref{j5}$_{1,5}$, we obtain
	\begin{align}\label{j6}
		A=\frac{\beta}{i \lambda \gamma - \rho \lambda^2 + \beta + \alpha \omega^2 \lambda^2}B, \quad E=-\frac{\omega \lambda}{i \lambda \tau + \sigma}D.
	\end{align}
	Substituting \eqref{j6} into equations \eqref{j5}$_{2,3,4}$, it follows that
	\begin{align}
		\underbrace{
			\left[ \begin{array}{ccc}
				q_{1}(\lambda) & -\kappa\omega\lambda & m\omega\lambda \\
				-\kappa\omega\lambda & q_{2}(\lambda) & -m \\
				-im\omega\lambda & im & q_{3}(\lambda)
			\end{array} \right]
		}_{:=M}
		\left[ \begin{array}{c}
			B \\
			C \\
			D
		\end{array} \right]
		=
		\left[ \begin{array}{c}
			0 \\
			1 \\
			0
		\end{array} \right],
	\end{align}
	where $q_1(\lambda)$, $q_2(\lambda)$, $q_3(\lambda)$ are defined as
	\begin{align}
		\begin{cases}
			q_1(\lambda) &= -\dfrac{\beta^2}{i \lambda \gamma - \rho \lambda^2 + \beta + \alpha \omega^2 \lambda^2} + \beta - \rho_1 \lambda^2 + \kappa \omega^2 \lambda^2 \\
&= -\dfrac{\beta^2}{i \lambda \gamma - \rho \lambda^2 + \beta + \alpha \omega^2 \lambda^2} + \beta + \kappa\left(\dfrac{\rho_2}{b}-\dfrac{\rho_1}{\kappa}\right) \lambda^2,\\
			q_2(\lambda)& = \kappa - \rho_2 \lambda^2 + b \omega^2 \lambda^2 = \kappa,\\
			q_3(\lambda) &= i \rho_3 + \dfrac{\omega^2 \lambda}{i \lambda \tau + \sigma} = i\left(\rho_3 - \dfrac{\tau \omega ^2 \lambda^2}{\sigma^2 + \tau^2 \lambda^2}\right) + \dfrac{\sigma \omega^2 \lambda}{\sigma^2 + \tau^2 \lambda^2}.
		\end{cases}
	\end{align}
	Therefore, the determinant of $M$ is
	\begin{align*}
		\det M &= -im\omega\lambda \left( \kappa\omega m\lambda - \kappa\omega m\lambda \right) - im(-mq_1(\lambda)+\kappa m\omega^2\lambda^2)  \\
&+ q_3(\lambda)(q_1(\lambda) q_2(\lambda) - \kappa^2 \omega^2\lambda^2) \\
		&= - im(-mq_1(\lambda)+\kappa m\omega^2\lambda^2)  + q_3(\lambda)(q_1(\lambda) q_2(\lambda) - \kappa^2 \omega^2\lambda^2) \\
		&= -im\left( \frac{m\beta^2}{i\lambda \gamma-\rho\lambda^2+\beta+\alpha\omega^2\lambda^2}  - m\beta + m \rho_1 \lambda^2\right) \\
		&\quad + \left(i\left(\rho_3 - \dfrac{\tau \omega ^2 \lambda^2}{\sigma^2 + \tau^2 \lambda^2}\right) + \dfrac{\sigma \omega^2 \lambda}{\sigma^2 + \tau^2 \lambda^2}\right)\\
&\quad\quad\times\left( -\frac{\kappa\beta^2}{i\lambda \gamma-\rho\lambda^2+\beta+\alpha \omega^2\lambda^2}+\kappa\beta-\kappa\rho_1 \lambda^2\right).
	\end{align*}
	Considering the real part of $\det M$, we conclude that $\det M \neq 0$. Thus, $C$ can be solved as
	\begin{align}\label{C}
		C = \dfrac{q_1(\lambda)q_3(\lambda) + im^2 \omega^2 \lambda^2}{\det M}.
	\end{align}
	\paragraph{Case 1.} If $\chi_0 \neq 0$ and $\chi_1 \neq 0$, as $n \to \infty$, we obtain
	\begin{align*}
		q_1(\lambda)q_3(\lambda) + im^2\omega^2\lambda^2 &\sim \kappa\left(\frac{\rho_2}{b}-\frac{\rho_1}{\kappa}\right)\lambda^2 \cdot i\left(\rho_3-\frac{\omega^2}{\tau}\right) + im^2\omega^2\lambda^2 \\
		&= i\left[\kappa\left(\frac{\rho_2}{b}-\frac{\rho_1}{\kappa}\right)\left(\rho_3-\frac{\omega^2}{\tau}\right) + m^2\omega^2\right]\lambda^2 \\
		&= i\chi_0\lambda^2,
	\end{align*}
	where $\chi_0 = \kappa\left(\dfrac{\rho_2}{b} -  \dfrac{\rho_1}{\kappa}\right)\left(\rho_3 - \dfrac{\rho_2}{b\tau}\right) + \dfrac{m^2 \rho_2}{b}$.
	\begin{align*}
		\det M &\sim -im^2 \rho_1 \lambda^2 + i \left(\rho_3 - \frac{\omega^2}{\tau}\right)(-\kappa \rho_1 \lambda^2) \\
		&= i \rho_1 \left(-m^2 - \kappa \rho_3 + \frac{\kappa \omega^2}{\tau}\right) \lambda^2\\
		&=i \rho_1 \chi_1 \lambda^2,
	\end{align*}
	where $\chi_1 = -m^2 - \kappa\rho_3 + \dfrac{\kappa\rho_2}{\tau b}$. From \eqref{C}, it follows that
	\begin{align*}
		|C| \approx \left|\dfrac{\chi_0}{\rho_1 \chi_1}\right| > 0.
	\end{align*}
	Since $\Psi = i \lambda \psi = i \lambda C \sin\omega \lambda x$, we derive
	\begin{align*}
		\|\Psi\|^2 = \lambda^2 C^2 \int_{0}^{L} \sin^2\omega\lambda x dx = \lambda^2 C^2 \int_{0}^{L} \sin^2\frac{n\pi}{L}x dx = \dfrac{LC^2\lambda^2}{2} \to \infty,
	\end{align*}
	which implies $\|U_n\|_{\mathcal{H}_1} \to \infty$.

	\paragraph{Case 2.} If $\chi_0 \neq 0$ and $\chi_1 = 0$, as $n \to \infty$, we obtain
	\begin{align*}
		\det M \sim \dfrac{\sigma \omega^2}{\tau^2} \dfrac{1}{\lambda}(-\kappa \rho_1 \lambda^2) = -\dfrac{\kappa \rho_1 \sigma \omega^2}{\tau^2} \lambda.
	\end{align*}
	From \eqref{C}, it follows that
	\begin{align*}
		|C| \approx \left|\dfrac{\tau^2 \chi_0}{\kappa \rho_1 \sigma \omega^2}\lambda\right| \to \infty,
	\end{align*}
	which implies that
	\begin{align*}
		\|\Psi\|^2 = \lambda^2 C^2 \int_{0}^{L} \sin^2\omega\lambda x dx = \lambda^2 C^2 \int_{0}^{L} \sin^2\frac{n\pi}{L}x dx = \dfrac{LC^2\lambda^2}{2} \to \infty,
	\end{align*}
	and hence $\|U_n\|_{\mathcal{H}_1} \to \infty$.

	\paragraph{Case 3.} If $\chi_0 = 0$ and $\chi_1 = 0$, since $\chi_0 = \kappa\left(\dfrac{\rho_2}{b} -  \dfrac{\rho_1}{\kappa}\right)\left(\rho_3 - \dfrac{\rho_2}{b\tau}\right) + \dfrac{m^2 \rho_2}{b} = 0$, we get $\dfrac{\rho_2}{b} - \dfrac{\rho_1}{\kappa} \neq 0$. As $n \to \infty$, we obtain
	\begin{align*}
		q_1(\lambda)q_3(\lambda) + im^2\omega^2\lambda^2 &\sim \kappa\left(\dfrac{\rho_2}{b} - \dfrac{\rho_1}{\kappa}\right) \lambda^2 \dfrac{\sigma\omega^2}{\tau^2} \dfrac{1}{\lambda} = \dfrac{\kappa\sigma\omega^2(\frac{\rho_2}{b} - \frac{\rho_1}{\kappa})}{\tau^2} \lambda \neq 0.
	\end{align*}
	From \eqref{C}, it follows that
	\begin{align*}
		|C| \approx \dfrac{1}{\rho_1}\left|\dfrac{\rho_2}{b} - \dfrac{\rho_1}{\kappa}\right| > 0,
	\end{align*}
	which implies $\|\Psi\|^2 \to \infty$. Therefore, $\|U_n\|_{\mathcal{H}_1} \to \infty$.

	The proof of Theorem \ref{Non-exponential Stability} is complete.
\end{proof}

\section{The Memory Case ($\delta = 1$): Exponential Stability} \label{sec3}
When $\delta = 1$, the suspension bridge model \eqref{model1.1} incorporates the Boltzmann memory term and takes the form in $(0,L) \times \mathbb{R}^+$:
\begin{align}\label{eq:with memory}
	\begin{cases}
		\rho u_{tt} - \alpha u_{xx} - \beta(\varphi - u) + \gamma u_t = 0, \\
		\rho_1 \varphi_{tt} - \kappa(\varphi_x + \psi)_x + \beta (\varphi - u) + m \theta_x = 0, \\
		\rho_2 \psi_{tt} - b \psi_{xx} + \kappa(\varphi_x + \psi) + \int_0^\infty g(s) \psi_{xx}(t-s)  ds - m \theta = 0, \\
		\rho_3 \theta_t + q_x + m (\varphi_x + \psi)_t = 0, \\
		\tau q_t + \sigma q + \theta_x = 0.
	\end{cases}
\end{align}

For the memory kernel $g(s)$ in the system, we impose the following standard assumptions, which are conventional in viscoelasticity research.

\begin{assumption}\label{assumg}
	The memory kernel $g \in L^1(\mathbb{R}^+) \cap C^1(\mathbb{R}^+)$ is a nonnegative function satisfying
	\begin{equation}\label{g}
		\dt b := b - b_0 > 0 \quad \text{with } b_0 = \int_{0}^{\infty} g(s) ds, \quad  g'(s) \leq -\kappa_1 g(s) \quad \text{for } s\in\mathbb{R}^+,
	\end{equation}
	where $\kappa_1 > 0$ is a given decay rate.
\end{assumption}

Define the relative history function
\begin{equation*}
	\eta^t(x,s) := \psi(x,t) - \psi(x,t-s) \quad \text{for } t,s \geq 0.
\end{equation*}
Using the relative history function, the original system \eqref{eq:with memory} is transformed into
\begin{align}
	\begin{cases}
		\rho u_{tt} - \alpha u_{xx} - \beta(\varphi - u) + \gamma u_t = 0 , \\
		\rho_1 \varphi_{tt} - \kappa(\varphi_x + \psi)_x + \beta (\varphi - u) + m \theta_x = 0, \\
		\rho_2 \psi_{tt} - \dt b \psi_{xx} + \kappa(\varphi_x + \psi) - \int_0^\infty g(s) \eta_{xx}(s)  ds - m \theta = 0, \\
		\rho_3 \theta_t + q_x + m (\varphi_x + \psi)_t = 0, \\
		\tau q_t + \sigma q + \theta_x = 0,\\
		\eta_t + \eta_s - \psi_t = 0,
	\end{cases}
\end{align}
for $s>0$, where the boundary satisfies
\begin{align}\label{with memory:boundary condition}
	\begin{cases}
		u_x(0,t) = u_x(L,t) = \varphi_x(0,t) = \varphi_x(L,t) =0\\
 \psi(0,t) = \psi(L,t) = \theta(0,t) = \theta(L,t) = 0, \\
		\eta^t(0,s) = \eta^t(L,s) = 0, \quad \eta^t(\cdot,0) = 0,
	\end{cases}
\end{align}
for $s, t \in \mathbb{R}^+$, and the initial values are defined as
\begin{equation}\label{with memory:initial condition}
	\begin{cases}
		u(x,0) = u_0(x), \quad u_t(x,0) = v_0(x), \\
		\varphi(x,0) = \varphi_0(x), \quad \varphi_t(x,0) = \Phi_0(x), \\
		\psi(x,0) = \psi_0(x), \quad \psi_t(x,0) = \Psi_0(x), \\
		\theta(x,0) = \theta_0(x), \quad q(x,0)=q_0(x), \\
		\eta^0(x,s) = \eta_0(x,s),
	\end{cases}
\end{equation}
for $x \in (0, L)$, $s \in \mathbb{R}^+$.

We introduce the following image space:
\[
\mathcal{H}_2 := H_*^1 \times L_*^2 \times H_*^1 \times L_*^2 \times H_0^1 \times L^2 \times L^2 \times L^2 \times L_g^2(\mathbb{R}^+; H_0^1),
\]
which forms a Hilbert space when endowed with the inner product
\begin{align*}
	(U_1, U_2)_{\mathcal{H}_2} := &\rho(v_1, v_2) + \alpha (u_{1x}, u_{2x}) + \rho_1(\Phi_1, \Phi_2) + \rho_2(\Psi_1, \Psi_2) + \dt b (\psi_{1x},\psi_{2x}) \\
	&+ \beta (\varphi_1 - u_1, \varphi_2 - u_2) + \kappa(\varphi_{1x} + \psi_1, \varphi_{2x} + \psi_2) + \rho_3(\theta_1, \theta_2) \\
&+ \tau(q_1, q_2) + (\eta_1, \eta_2)_{L_g^2(\mathbb{R}^+; H_0^1)},
\end{align*}
and the norm
\begin{align*}
\|U\|_{\mathcal{H}_2}^2 :=& \rho\|v\|^2 + \alpha\|u_x\|^2 + \rho_1\|\Phi\|^2 + \rho_2\|\Psi\|^2 + \dt b\|\psi_x\|^2 + \beta\|\varphi - u\|^2 \\
&+ \kappa\|\varphi_x + \psi\|^2 + \rho_3\|\theta\|^2 + \tau \|q\|^2 + \|\eta\|_{L_g^2(\mathbb{R}^+; H_0^1)}^2,
\end{align*}
where $$U := (u, v, \varphi, \Phi, \psi, \Psi, \theta, q, \eta) \in \mathcal{H}_2, U_i := (u_i, v_i, \varphi_i, \Phi_i, \psi_i, \Psi_i, \theta_i, q_i, \eta_i) \in \mathcal{H}_2$$ for $i = 1, 2$. Thus, the system can be transformed into the following Cauchy problem:
\begin{equation}\label{model-abstract-2}
	\begin{cases}
		\dfrac{d}{dt}U = \mathcal{A}_2U, & t > 0, \\
		U(0) =  U_0 :=(u_0, v_0, \varphi_0, \Phi_0, \psi_0, \Psi_0, \theta_0, q_0, \eta_0),
	\end{cases}
\end{equation}
where $\mathcal{A}_2: \mathcal{D}(\mathcal{A}_2) \subset \mathcal{H}_2 \to \mathcal{H}_2$ is a linear operator defined as
\begin{equation}\label{definition:A-2}
	\mathcal{A}_2 U := \begin{pmatrix}
		v \\
		\rho^{-1} \left[ \alpha u_{xx} + \beta (\varphi - u) - \gamma v \right] \\
		\Phi \\
		\rho_1^{-1} \left[ \kappa (\varphi_x + \psi)_x - \beta (\varphi - u) - m \theta_x \right] \\
		\Psi \\
		\rho_2^{-1} \left[ \dt b \psi_{xx} - \kappa (\varphi_x + \psi) + \int_0^\infty g(s) \eta_{xx}(s) ds + m\theta \right] \\
		\rho_3^{-1} \left[ -q_x - m (\Phi_x + \Psi) \right] \\
		\tau^{-1}(-\sigma q - \theta_x) \\
		\Psi - \eta_s
	\end{pmatrix}.
\end{equation}
The domain of $\mathcal{A}_2$ is defined as
\begin{align*}
	\mathcal{D}(\mathcal{A}_2) := \bigg\{ U \in \mathcal{H}_2 :&
	v, \Phi \in H_*^1; \ u_x, \varphi_x, \theta, \Psi \in H_0^1; \\
	&\dt b \psi + \int_0^\infty g(s)\eta(s)ds \in H^2; \\
	&q \in H^1; \
	\eta_s \in L_g^2(\mathbb{R}^+; H_0^1); \ \eta(\cdot, 0) = 0
	\bigg\}.
\end{align*}
The main results regarding well-posedness and exponential stability are as follows.

\begin{theorem}[Well-posedness]\label{Well-Posedness-2}
	Let \(\mathcal{A}_2\) be the linear operator defined in \eqref{definition:A-2} with domain \(\mathcal{D}(\mathcal{A}_2) \subset \mathcal{H}_2\). Then \(\mathcal{A}_2\) is the infinitesimal generator of a \(C_0\)-contraction semigroup \(\{e^{t\mathcal{A}_2}\}_{t \ge 0}\) on \(\mathcal{H}_2\). Consequently:
	If \(U_0 \in \mathcal{H}_2\), the abstract Cauchy problem \eqref{model-abstract-2} admits a unique mild solution \(U(t) = e^{t\mathcal{A}_2}U_0\) satisfying \(U \in C(\mathbb{R}^+, \mathcal{H}_2)\); if \(U_0 \in \mathcal{D}(\mathcal{A}_2)\), then \(U(t) = e^{t\mathcal{A}_2}U_0\) is the unique classical solution to \eqref{model-abstract-2}, and it belongs to \(C(\mathbb{R}^+, \mathcal{D}(\mathcal{A}_2)) \cap C^1(\mathbb{R}^+, \mathcal{H}_2)\).
\end{theorem}

\begin{theorem}[Exponential stability]\label{exponential stability}
	Under Assumption \ref{assumg}, the \(C_0\)-semigroup \(\{e^{\mathcal{A}_2 t}\}_{t \geq 0}\) generated by \(\mathcal{A}_2\) is exponentially stable. Specifically, there exist positive constants \(\alpha, \beta > 0\), independent of \(t\), such that
	\[
	\|e^{\mathcal{A}_2 t}\|_{\mathcal{L}(\mathcal{H}_2 \to \mathcal{H}_2)} \leq \alpha e^{-\beta t} \quad \text{for all } t \ge 0.
	\]
	Moreover, the decay rate can be taken as $\beta = \min\{\gamma, \sigma, \kappa_1/2\} / C$ for some constant $C > 0$ depending only on the structural parameters.
\end{theorem}

\subsection{Well-Posedness}\label{sec3.1}
Following the approach in Subsection \ref{sec2.1}, to prove Theorem \ref{Well-Posedness-2}, we only need to show that $\mathcal{A}_2$ is both dissipative and maximal, as stated in the following two lemmas.

\begin{lemma}\label{dissipative-2}
	The operator $\mathcal{A}_2$ defined in \eqref{definition:A-2} is dissipative.
\end{lemma}
\begin{proof}
	For any $U \in \mathcal{D}(\mathcal{A}_2)$, using Assumption \ref{assumg}, a direct calculation gives
	\begin{align*}
		{\rm Re}(\mathcal{A}_2U,U)_{\mathcal{H}_2} &= {\rm Re}\bigg\{
		\rho\left(\rho^{-1} \left[ \alpha u_{xx} + \beta (\varphi - u) - \gamma v \right], v\right) + \alpha\left(v_x, u_x\right)  \\
&\quad + \rho_1\left(\rho_1^{-1} \left[ \kappa (\varphi_x + \psi)_x - \beta (\varphi - u) - m \theta_x \right], \Phi\right) \\
		&\quad + \rho_2\left(\rho_2^{-1} \left[\dt b \psi_{xx} - \kappa (\varphi_x + \psi) + \int_0^\infty g(s) \eta_{xx}(s) ds + m\theta \right], \Psi\right) \\
&\quad + \dt b\left(\Psi_x,\psi_x\right) + \beta\left(\Phi - v, \varphi - u\right)  \\
		&\quad + \kappa\left(\Phi_x + \Psi, \varphi_x + \psi\right)  + \rho_3\left(\rho_3^{-1} \left[ -q_x - m (\Phi_x + \Psi) \right], \theta\right)  \\
		&\quad +\tau\left(\tau^{-1}(-\sigma q - \theta_x),q\right) + \int_0^\infty g(s)(\Psi_x - \eta_{sx}, \eta_x)  ds \bigg\} \\
		&= -\gamma\|v\|^2  - \sigma \|q\|^2 + \frac{1}{2}\int_{0}^{\infty}g'(s)\|\eta_x\|^2ds\\
		&\le -\gamma\|v\|^2  - \sigma \|q\|^2 - \frac{\kappa_1}{2} \int_0^\infty g(s) \|\eta_x\|^2 ds\\
		&\le 0.
	\end{align*}
	Thus, $\mathcal{A}_2$ satisfies the dissipativity condition.
\end{proof}

\begin{lemma}\label{maximal-2}
	$0 \in \varrho(\mathcal{A}_2)$. Consequently, $\mathcal{A}_2$ is maximal.
\end{lemma}
\begin{proof}
	For $\mathcal{A}_2U = F:=(f_1,f_2,f_3,f_4,f_5,f_6,f_7,f_8,f_9) \in \mathcal{H}_2$, we aim to prove that there exists a unique solution $U\in\mathcal{D}(\mathcal{A}_2)$ satisfying $\|U\|_{\mathcal{H}_2} \leq C\|F\|_{\mathcal{H}_2}$ for some constant $C > 0$.

	Substituting the definition of $\mathcal{A}_2$ into $\mathcal{A}_2U = F$, we obtain
	\begin{equation}\label{eq:resolvent-system-A-2}
		\begin{cases}
			v = f_1, \\
			\alpha u_{xx} + \beta (\varphi - u) - \gamma v = \rho f_2, \\
			\Phi = f_3, \\
			\kappa (\varphi_x + \psi)_x - \beta (\varphi - u) - m \theta_x = \rho_1 f_4, \\
			\Psi = f_5, \\
			\dt b \psi_{xx} - \kappa (\varphi_x + \psi) + \int_0^\infty g(s) \eta_{xx}(s) ds + m\theta = \rho_2 f_6, \\
			-q_x - m (\Phi_x + \Psi) = \rho_3 f_7, \\
			-\sigma q - \theta_x = \tau f_8, \\
			\Psi - \eta_s = f_9.
		\end{cases}
	\end{equation}
	First, equations \eqref{eq:resolvent-system-A-2}$_{1,3,5}$ yield
	\begin{equation}\label{df1}
		v = f_1 \in H_*^1, \quad \Phi = f_3 \in H_*^1, \quad \Psi = f_5 \in H_0^1.
	\end{equation}
	Second, from equation \eqref{eq:resolvent-system-A-2}$_9$, it follows that $\eta_s = \Psi - f_9 \in L_g^2(\mathbb{R}^+; H_0^1)$. Integrating with respect to $s$ yields
	\begin{align}\label{def:eta}
		\eta(s) = s\Psi - \int_0^s f_9(\tau)d\tau\in C(\mathbb{R}^+; H_0^1).
	\end{align}
	Now we prove that $\eta \in L_g^2(\mathbb{R}^+; H_0^1)$. First, we show that $\int_0^N g(s)\|\eta_x\|^2ds$ is bounded for any $N > 0$. By \eqref{def:eta}, the following estimate holds:
	\begin{align*}
		\|\eta_x\|^2 &= \|s\Psi_x - \int_0^s f_{9x}(\tau)d\tau\|^2\\
		 &\le 2s^2\|\Psi_x\|^2 + 2s\int_0^s\| f_{9x}(\tau)\|^2d\tau  \\
		&\le 2s^2\|\Psi_x\|^2 + \frac{2s}{g(s)}\int_0^\infty g(\tau)\| f_{9x}(\tau)\|^2d\tau \\
		&= 2s^2\|\Psi_x\|^2 + \frac{2s}{g(s)}\|f_9\|_{L_g^2(\mathbb{R}^+; H_0^1)} ^2.
	\end{align*}
	Combined with Assumption \ref{assumg}, we conclude that $g$ decays exponentially. Therefore, for any $N > 0$, there exists a constant $C_N$ depending on $N$ such that
	\begin{align}
		\int_0^N g(s)\|\eta_x\|^2ds &\leq 2\int_0^N g(s)s^2ds \|\Psi_x\|^2 + 2\int_0^N s ds \|f_9\|_{L_g^2(\mathbb{R}^+; H_0^1)}^2\notag \\
		&\leq C_N.\label{ddf}
	\end{align}
	Now, using Assumption \ref{assumg}, integrating by parts, and applying Young's inequality, it follows that
	\begin{align*}
		\int_0^N g(s)\|\eta_x\|^2ds &\leq -\frac{1}{\kappa_1}\int_0^N g'(s)\|\eta_x\|^2ds \\
		&= -\frac{g(N)}{\kappa_1}\|\eta_x(N)\|^2 \!+ \!\frac{g(0)}{\kappa_1}\|\eta_x(0)\|^2\! +\! \frac{2}{\kappa_1}\int_0^N g(s){\rm Re}\left(\eta_x, \eta_{sx}\right)ds \\
		&\leq \frac{2}{\kappa_1}\int_0^N g(s)\|\eta_x\|\|\eta_{sx}\|ds \\
		&\leq \frac{1}{2}\int_0^N g(s)\|\eta_x\|^2ds + \frac{2}{\kappa_1^2}\int_0^N g(s)\|\eta_{sx}\|^2ds.
	\end{align*}
	Since by estimate \eqref{ddf}, $\int_0^N g(s)\|\eta_x\|^2ds$ is bounded, the terms in the above inequality can be rearranged to yield
	\begin{align*}
		\int_0^N g(s)\|\eta_x\|^2ds \leq \frac{4}{\kappa_1^2}\int_0^N g(s)\|\eta_{sx}\|^2ds.
	\end{align*}
	Letting $N \to \infty$, we find that $\eta \in L_g^2(\mathbb{R}^+; H_0^1)$ follows directly from $\eta_s \in L_g^2(\mathbb{R}^+; H_0^1)$.

	Using equation \eqref{eq:resolvent-system-A-2}$_8$, substituting $q$ into equation \eqref{eq:resolvent-system-A-2}$_7$, and rearranging equations \eqref{eq:resolvent-system-A-2}$_{2,4,6}$, we derive
	\begin{align}\label{eq:resolvent-system-1-A-2}
		\begin{cases}
			-\alpha u_{xx} - \beta(\varphi - u) = -\rho f_2 - \gamma v=-\rho f_2 - \gamma f_1 := g_1 \in L_*^2, \\
			-\kappa(\varphi_x + \psi)_x + \beta(\varphi - u) + m\theta_x = -\rho_1 f_4 := g_2 \in L_*^2, \\
			- \dt b \psi_{xx} + \kappa(\varphi_x + \psi) - m\theta \\
\quad\quad\quad= -\rho_2 f_6 +  \int_0^\infty g(s) \eta_{xx}(s) ds  := g_3 \in H^{-1}, \\
			- \theta_{xx} = - \sigma m(\Phi_x + \Psi) - \sigma \rho_3 f_7 + \tau f_{8x} \\
\quad\quad\quad= - \sigma m(f_{3x} + f_5) - \sigma \rho_3 f_7 + \tau f_{8x}:=g_4 \in H^{-1}.
		\end{cases}
	\end{align}
	We now sketch how the method of Lemma \ref{maximal-1} extends to this system. The key difference is the additional memory term $\int_0^\infty g(s)\eta_{xx}(s)ds$ in the third equation, which belongs to $H^{-1}$ since $\eta \in L_g^2(\mathbb{R}^+; H_0^1)$ and $g \in L^1(\mathbb{R}^+)$. In the weak formulation, this term contributes an additional bilinear form
	\[
	\mathcal{L}_{\mathrm{mem}}\left(\begin{pmatrix}u\\ \varphi \\ \psi \\ \theta\end{pmatrix}, \begin{pmatrix}\tilde{u}\\ \tilde{\varphi} \\ \tilde{\psi} \\ \tilde{\theta}\end{pmatrix}\right) = \int_0^\infty g(s)(\eta_x(s), \tilde{\psi}_x)ds,
	\]
	where $\eta$ is expressed in terms of $\psi$ and the data via \eqref{def:eta}. Since $\eta_x = s\psi_x - \int_0^s f_{9x}(\tau)d\tau$, we obtain
	\[
	|\mathcal{L}_{\mathrm{mem}}| \leq \left(\int_0^\infty g(s)s^2ds\right)^{1/2}\|\psi_x\| \|\tilde{\psi}_x\| + \left(\int_0^\infty g(s)ds\right)^{1/2}\|f_9\|_{L_g^2}\|\tilde{\psi}_x\|.
	\]
	The first term is absorbed by the coercivity of $\dt b\|\psi_x\|^2$ thanks to Assumption \ref{assumg} ($\dt b = b - b_0 > 0$), while the second term is controlled by $\|F\|_{\mathcal{H}_2}\|\tilde{\psi}_x\|$. Hence the bilinear form remains coercive and the Lax--Milgram argument applies verbatim, yielding a unique solution $(u,\varphi,\psi,\theta) \in H_*^1 \times H_*^1 \times H_0^1 \times H_0^1$. Furthermore, we can solve for $q$, demonstrate that $U \in \mathcal{D}(\mathcal{A}_2)$, and similarly obtain $\|U\|_{\mathcal{H}_2} \leq C\|F\|_{\mathcal{H}_2}$. Consequently, $0 \in \varrho(\mathcal{A}_2)$.
\end{proof}

\subsection{Exponential Stability}\label{sec3.2}
In this section, we prove the exponential stability result, i.e., Theorem \ref{exponential stability}. The proof is carried out in two sequential steps, as detailed below. We invoke the following lemma.

\begin{lemma}\label{Thm:SpectralApprox}\cite{MR933321}
	Let $\mathcal{A}$ be the infinitesimal generator of a $C_0$-semigroup on a reflexive Banach space $X$. Then
	$$
	\sigma(\mathcal{A}) \cap i\mathbb{R} = \sigma_{ap}(\mathcal{A}) \cap i\mathbb{R},
	$$
	where $\sigma(\mathcal{A})$ and $\sigma_{ap}(\mathcal{A})$ denote the spectrum and approximate spectrum of the operator $\mathcal{A}$, respectively, and $\sigma_{ap}(\mathcal{A})$ is defined as
	\begin{align}\label{defapp}
		\sigma_{ap}(\mathcal{A}):= \{ \lambda \in \mathbb{C}: \exists U_{n} \in \mathcal{D}(\mathcal{A}), \|U_{n}\|_{\mathcal{H}} = 1, (\lambda I - \mathcal{A})U_{n} \to 0 \text{ in } \mathcal{H}\}.
	\end{align}
\end{lemma}

\textbf{Step 1: The resolvent set of $\mathcal{A}_2$ contains the imaginary axis.}
To prove that $i\mathbb{R} \subset \rho(\mathcal{A}_2)$, we employ a contradiction argument. Suppose, for contradiction, that $i\mathbb{R} \not\subset \rho(\mathcal{A}_2)$. Since $0 \in \rho(\mathcal{A}_2)$ by Lemma \ref{maximal-2}, we assume without loss of generality that there exists $\omega > 0$ such that $i\omega \notin \rho(\mathcal{A}_2)$. Thus, $i\omega \in \sigma(\mathcal{A}_2)$. Then by Lemma \ref{Thm:SpectralApprox}, $i\omega \in \sigma_{ap}(\mathcal{A}_2)$. From the definition of $\sigma_{ap}(\mathcal{A}_2)$, it follows that there exists a sequence $\{U_n=(u_n, v_n, \varphi_n, \Phi_n, \psi_n, \Psi_n, \theta_n, q_n, \eta_n) \} \subset \mathcal{D}(\mathcal{A}_2)$ with $\|U_n\|_{\mathcal{H}_2} = 1$ such that
\begin{align}\label{convergence}
	(i\omega I - \mathcal{A}_2)U_n \to 0 ~~~~~~~~~\text{in} ~\mathcal{H}_2.
\end{align}
Using the definition of $\mathcal{A}_2$ in \eqref{definition:A-2}, convergence \eqref{convergence} yields
\begin{align}\label{omega-limit-system}
	\begin{cases}
		i\omega u_n - v_n  \to 0 \quad &\text{in } H_*^1,\\
		i\omega \rho v_n -\alpha u_{nxx}-\beta(\varphi_n - u_n)+ \gamma v_n \to 0 \quad &\text{in } L_*^2,\\
		i\omega\varphi_n - \Phi_n \to 0\quad &\text{in } H_*^1,\\
		i\omega\rho_1 \Phi_n - \kappa(\varphi_{nx}+\psi_n)_x \\
\quad\quad\quad+ \beta (\varphi_n - u_n) +  m\theta_{nx} \to 0 \quad &\text{in } L_*^2,\\
		i\omega \psi_n - \Psi_n \to 0 \quad &\text{in } H_0^1,\\
		i\omega \rho_2 \Psi_n - \dt b \psi_{nxx} + \kappa(\varphi_{nx} + \psi_n) \\
\quad\quad\quad- \int_{0}^{\infty} g(s) \eta_{nxx}(s) ds - m\theta_n \to 0 \quad &\text{in } L^2,\\
		i\omega \rho_3 \theta_n + q_{nx} + m(\Phi_{nx} + \Psi_n) \to 0 \quad &\text{in } L^2,\\
		i\omega\tau q_n + \sigma q_n + \theta_{nx} \to 0 \quad &\text{in } L^2,\\
		i\omega \eta_n - \Psi_n + \eta_{ns} \to 0\quad &\text{in } L_g^2(\mathbb{R}^+;H_0^1).
	\end{cases}
\end{align}

\begin{lemma}\label{iR:vn,qn,etan}
	Under Assumption \ref{assumg}, as $n \to \infty$, we have
	\begin{align}
		\|v_n\| \to 0,\label{limit:vn}\\
		\|q_n\| \to 0,\label{limit:qn}\\
		\int_{0}^{\infty} g'(s)\|\eta_{nx}(s)\|^2ds \to 0,\label{limit:g'}\\
		\int_{0}^{\infty} g(s)\|\eta_{nx}(s)\|^2ds \to 0.\label{limit:g}
	\end{align}
\end{lemma}
\begin{proof}
	Taking the real part of the $\mathcal{H}_2$-inner product of $i \omega U_n - \mathcal{A}_2 U_n$ with $U_n$, we obtain
	\begin{align*}
		\operatorname{Re} ( i \omega U_n - \mathcal{A}_2 U_n, U_n )_{\mathcal{H}_2}
		&= -\operatorname{Re} ( \mathcal{A}_2 U_n, U_n )_{\mathcal{H}_2}\notag\\
		&= \gamma\|v_n\|^2 + \sigma\|q_n\|^2 - \frac{1}{2}\int_{0}^{\infty} g'(s)\|\eta_{nx}\|^2ds.
	\end{align*}
	From convergence \eqref{convergence} and $\|U_n\|_{\mathcal{H}_2} = 1$, we have
	\begin{align*}
		\operatorname{Re} ( i \omega U_n - \mathcal{A}_2 U_n, U_n )_{\mathcal{H}_2} \to 0.
	\end{align*}
	Thus, \eqref{limit:vn}--\eqref{limit:g'} follow immediately. Estimate \eqref{limit:g} is derived from \eqref{limit:g'} using Assumption \ref{assumg}.
\end{proof}

\begin{lemma}\label{iR:thetan}
	Under Assumption \ref{assumg}, as $n \to \infty$, we have
	\begin{align}
		\|\theta_{nx}\| \to 0.\label{limit:thetan}
	\end{align}
\end{lemma}
\begin{proof}
	From \eqref{omega-limit-system}$_8$ and \eqref{limit:qn}, it follows that
	\begin{align*}
		\|\theta_{nx}\| \leq \|i\omega\tau q_n + \sigma q_n + \theta_{nx}\| + \|i\omega\tau q_n + \sigma q_n\| \to 0.
	\end{align*}
\end{proof}

\begin{lemma}\label{iR:un}
	Under Assumption \ref{assumg}, as $n \to \infty$, we have
	\begin{align}
		\|u_{nx}\| \to 0.\label{limit:un}
	\end{align}
\end{lemma}
\begin{proof}
	By \eqref{omega-limit-system}$_1$ and \eqref{limit:vn}, we obtain
	\begin{align*}
		\omega\|u_n\| = \|i\omega u_n\| \leq \|i\omega u_n - v_n\| + \|v_n\| \to 0.
	\end{align*}
	Thus $\|u_n\| \to 0$. Taking the $L^2$-inner product of \eqref{omega-limit-system}$_2$ with $u_n$ and integrating by parts, we obtain
	\begin{align*}
		(i\omega\rho v_n - \beta(\varphi_n - u_n) + \gamma v_n, u_n) + \alpha\|u_{nx}\|^2 \rightarrow 0.
	\end{align*}
	The first term satisfies
	\begin{align*}
		|(i\omega\rho v_n - \beta(\varphi_n - u_n) + \gamma v_n, u_n)| &\leq (( \omega\rho + \gamma ) \|v_n\| + \beta \|\varphi_n - u_n\|) \|u_n\| \\&\leq c \|u_n\| \rightarrow 0.
	\end{align*}
	Therefore, $\|u_{nx}\| \to 0$.
\end{proof}

\begin{lemma}\label{iR:Phin}
	Under Assumption \ref{assumg}, as $n \to \infty$, we have
	\begin{align}
		\|\Phi_n\| \to 0.\label{limit:Phin}
	\end{align}
\end{lemma}
\begin{proof}
	Taking the $L^2$-inner product of \eqref{omega-limit-system}$_2$ with $\varphi_n$ and integrating by parts, we obtain
	\begin{align}\label{jvarphin}
		(i \omega \rho v_n + \beta u_n + \gamma v_n, \varphi_n) + \alpha(u_{nx}, \varphi_{nx}) - \beta \|\varphi_n\|^2 \to 0.
	\end{align}
	The sum of the first two terms satisfies
	\begin{align*}
		&|(i \omega \rho v_n + \beta u_n + \gamma v_n, \varphi_n) + \alpha(u_{nx}, \varphi_{nx})| \\
		&\leq (( \omega \rho + \gamma) \|v_n\| + \beta \|u_n\|) \|\varphi_n\| + \alpha \|u_{nx}\| \|\varphi_{nx}\| \\
		&\leq C(\|v_n\| + \|u_n\| + \|u_{nx}\|) \\
		&\to 0.
	\end{align*}
	Returning to \eqref{jvarphin}, we obtain $\|\varphi_n\| \to 0$. Combined with \eqref{omega-limit-system}$_3$, it follows that $\|\Phi_n\| \to 0$.
\end{proof}

\begin{lemma}\label{iR:varphin}
	Under Assumption \ref{assumg}, as $n \to \infty$, we have
	\begin{align}
		\|\varphi_{nx} + \psi_n\| \to 0.\label{limit:varphin}
	\end{align}
\end{lemma}
\begin{proof}
	Using $\|\Phi_n\|, \|\varphi_n\|, \|u_n\|, \|\theta_{nx}\| \to 0$, from \eqref{omega-limit-system}$_4$, we get
	\begin{align*}
		\|(\varphi_{nx} + \psi_n)_x\| \to 0.
	\end{align*}
	Since $\varphi_{nx} + \psi_n \in H_0^1$, Poincar\'e's inequality yields
	\begin{align*}
		\|\varphi_{nx} + \psi_n\| \leq \sqrt{c_p}\|(\varphi_{nx} + \psi_n)_x\| \to 0.
	\end{align*}
\end{proof}

\begin{lemma}\label{iR:psin,Psin}
	Under Assumption \ref{assumg}, as $n \to \infty$, we have
	\begin{align}
		\|\psi_{nx}\| \to 0,\label{limit:psin}\\
		\|\Psi_{nx}\| \to 0.\label{limit:Psin}
	\end{align}
\end{lemma}
\begin{proof}
	Taking the $L^2$-inner product of \eqref{omega-limit-system}$_6$ with $\psi_n$ and integrating by parts, we obtain
	\begin{align}\label{eq:psin}
		i\omega\rho_2(\Psi_n, \psi_n) + \dt b\|\psi_{nx}\|^2&+ (\kappa(\varphi_{nx} + \psi_n) - m\theta_n, \psi_n)  \notag\\
&+ \left(\int_{0}^{\infty} g(s)\eta_{nx}(s)ds,  \psi_{nx}\right) \to 0.
	\end{align}
	First, estimate the term $(\kappa(\varphi_{nx} + \psi_n) - m\theta_n, \psi_n)$:
	\begin{align*}
		\left|(\kappa(\varphi_{nx} + \psi_n) - m\theta_n, \psi_n)\right| &\leq (\kappa \|\varphi_{nx} + \psi_n\|  + |m|\|\theta_n\|)\|\psi_{n}\|\\
		&\leq C(\|\varphi_{nx} + \psi_n\| + \|\theta_n\|)
		\to 0.
	\end{align*}
	Second, estimate the term $\left(\int_{0}^{\infty} g(s)\eta_{nx}(s)ds,  \psi_{nx}\right)$:
	\begin{align*}
		\left|\left(\int_{0}^{\infty} g(s)\eta_{nx}(s)ds,  \psi_{nx}\right)\right| &\leq \left\|\int_{0}^{\infty} g(s)\eta_{nx}(s)ds\right\| \|\psi_{nx}\|\\
		&\leq b_0^{\frac{1}{2}} \|\eta_n\|_{L^2_g(\mathbb{R}^+; H_0^1)} \|\psi_{nx}\|
		\leq C \|\eta_n\|_{L^2_g(\mathbb{R}^+; H_0^1)}
		\to 0.
	\end{align*}
	Third, estimate the term $(\Psi_n, \psi_n)$. Taking the $L^2$-inner product of \eqref{omega-limit-system}$_7$ with $\psi_n$, we obtain
	\begin{align}\label{eq:Psin,psin}
		i\omega\rho_3(\theta_n, \psi_n) - (q_n, \psi_{nx}) - m(\Phi_n, \psi_{nx}) + m(\Psi_n, \psi_n) \to 0.
	\end{align}
	From \eqref{eq:Psin,psin}, \eqref{limit:thetan}, \eqref{limit:qn}, and \eqref{limit:Phin}, we derive
	\begin{align*}
		|m(\Psi_n, \psi_n)| &\leq |i\omega\rho_3(\theta_n, \psi_n) - (q_n, \psi_{nx}) - m(\Phi_n, \psi_{nx}) + m(\Psi_n, \psi_n)|  \\
&\quad+ |i\omega\rho_3(\theta_n, \psi_n) - (q_n, \psi_{nx}) - m(\Phi_n, \psi_{nx})| \\
		&\leq |i\omega\rho_3(\theta_n, \psi_n) - (q_n, \psi_{nx}) - m(\Phi_n, \psi_{nx}) + m(\Psi_n, \psi_n)|  \\
&\quad + \omega\rho_3 \|\theta_n\|\|\psi_n\| + \|q_n\|\|\psi_{nx}\| + |m|\|\Phi_n\|\|\psi_{nx}\| \\
		&\to 0.
	\end{align*}
	Returning to \eqref{eq:psin} yields $\|\psi_{nx}\| \to 0$. $\|\Psi_{nx}\| \to 0$ is a direct consequence of \eqref{omega-limit-system}$_5$ and \eqref{limit:psin}.
\end{proof}

From Lemmas \ref{iR:vn,qn,etan}--\ref{iR:psin,Psin}, it follows that $\|U_n\|_{\mathcal{H}_2} \to 0$, contradicting $\|U_n\|_{\mathcal{H}_2} = 1$. Thus, $i \mathbb{R} \subset \varrho(\mathcal{A}_2)$.

\textbf{Step 2: Uniform boundedness of $(i\lambda I-\mathcal{A}_2)^{-1}$.}
Consider the resolvent equation
\begin{align}\label{resolvent eq}
	(i\lambda I-\mathcal{A}_2)U=F := (f_1,f_2,f_3,f_4,f_5,f_6,f_7,f_8,f_9) \in \mathcal{H}_2,
\end{align}
where $U=(u,v,\varphi,\Phi,\psi,\Psi,\theta,q,\eta) \in \mathcal{D}(\mathcal{A}_2)$. From the definition of $\mathcal{A}_2$ in \eqref{definition:A-2}, equation \eqref{resolvent eq} gives
\begin{align}\label{Resolvent equations-2}
	\begin{cases}
		i\lambda u - v=f_1   \quad &\text{in } H_*^1,\\
		i\lambda \rho v -\alpha u_{xx}-\beta(\varphi - u)+ \gamma v =\rho f_2 \quad &\text{in } L_*^2,\\
		i\lambda\varphi - \Phi =f_3\quad &\text{in } H_*^1,\\
		i\lambda\rho_1 \Phi - \kappa(\varphi_{x}+\psi)_x + \beta (\varphi - u) +  m\theta_{x}=\rho_1f_4  \quad &\text{in } L_*^2,\\
		i\lambda \psi - \Psi =f_5 \quad &\text{in } H_0^1,\\
		i\lambda \rho_2 \Psi - \dt b \psi_{xx} + \kappa(\varphi_{x} + \psi) \\
\quad\quad\quad- \int_{0}^{\infty} g(s) \eta_{xx}(s) ds - m\theta =\rho_2f_6 \quad &\text{in } L^2,\\
		i\lambda \rho_3 \theta + q_x + m(\Phi_{x} + \Psi) =\rho_3f_7 \quad &\text{in } L^2,\\
		i\lambda \tau q + \sigma q + \theta_x = \tau f_8 \quad &\text{in } L^2, \\
		i\lambda \eta + \eta_s - \Psi=f_9 \quad &\text{in } L_g^2(\mathbb{R}^+;H_0^1).
	\end{cases}
\end{align}

\begin{lemma}\label{boundedness:v,q,eta}
	Under Assumption \ref{assumg}, the following estimates hold:
	\begin{align}
		\|v\|^2 \leq \frac{1}{\gamma}\|U\|_{\mathcal{H}_2}\|F\|_{\mathcal{H}_2}, \label{v}\\
		\|q\|^2 \leq \frac{1}{\sigma}\|U\|_{\mathcal{H}_2}\|F\|_{\mathcal{H}_2}, \label{q}\\
		-\int_{0}^{\infty}g'(s)\|\eta_x(s)\|^2ds \leq 2\|U\|_{\mathcal{H}_2}\|F\|_{\mathcal{H}_2},\label{g'}\\
		\|\eta\|_{L_g^2(\mathbb{R}^+; H_0^1) }^2 \leq \frac{2}{\kappa_1}\|U\|_{\mathcal{H}_2}\|F\|_{\mathcal{H}_2}.\label{g}
	\end{align}
\end{lemma}
\begin{proof}
	Taking the $\mathcal{H}_2$-inner product of $(i\lambda I-\mathcal{A}_2)U=F$ with $U$, we derive
	\begin{align*}
		\|F\|_{\mathcal{H}_2}\|U\|_{\mathcal{H}_2}
		&\geq |\operatorname{Re}(F,U)_{\mathcal{H}_2}| = |\operatorname{Re}(i\lambda U-\mathcal{A}_2U,U)_{\mathcal{H}_2}| \\
		& = |-\operatorname{Re}(\mathcal{A}_2U,U)_{\mathcal{H}_2}| = \gamma \|v\|^2 + \sigma \|q\|^2 - \frac{1}{2}\int_{0}^{\infty} g'(s)\|\eta_x(s)\|^2ds.
	\end{align*}
	Estimates \eqref{v}--\eqref{g'} follow immediately. Estimate \eqref{g} is derived from \eqref{g'} using Assumption \ref{assumg}.
\end{proof}

Following the same procedure as in the proofs of Lemmas \ref{lemma-ux}--\ref{lemma-Phi}, we obtain the following lemma.

\begin{lemma}\label{boundedness:ux}
	Under Assumption \ref{assumg}, for any $\epsilon>0$, there exists a constant $C_{\epsilon}>0$, independent of $\lambda$, such that
	\begin{align}
		\|u_x\|^2, \|\theta\|^2, \|\varphi_x + \psi\|^2, \|\varphi - u\|^2, \|\Phi\|^2 \leq \epsilon\|U\|_{\mathcal{H}_2}^2 + C_{\epsilon}\|F\|_{\mathcal{H}_2}^2
	\end{align}
	for $|\lambda| \geq 1$.
\end{lemma}

\begin{lemma}\label{boundedness:Psi}
	Under Assumption \ref{assumg}, for any $\epsilon>0$, there exists a constant $C_{\epsilon}>0$, independent of $\lambda$, such that
	\begin{align}
		 \|\Psi\|^2 \leq \epsilon\|U\|_{\mathcal{H}_2}^2 + C_{\epsilon}\|F\|_{\mathcal{H}_2}^2
	\end{align}
	for $|\lambda| \geq 1$.
\end{lemma}
\begin{proof}
	Taking the $L^2$-inner product of \eqref{Resolvent equations-2}$_6$ with $\Psi$ and integrating by parts, we obtain
	\begin{align}\label{hh}
		i\lambda\rho_2 \|\Psi\|^2 &+ \dt b(\psi_x, \Psi_x) + \kappa(\varphi_x + \psi, \Psi)\notag\\
& + \int_{0}^{\infty} g(s)(\eta_x, \Psi_x) ds - m(\theta, \Psi) = \rho_2 (f_6, \Psi).
	\end{align}
	Multiplying both sides of the above equation by $ g(s) $ and integrating with respect to $ s $ from $(0,\infty)$, we obtain
	\begin{align}\label{eq:Psi}
		 i\lambda\rho_2 b_0 \|\Psi\|^2 &+ \dt b(\psi, \Psi)_{L_g^2(\mathbb{R}^+; H_0^1)} + \kappa b_0(\varphi_x + \psi, \Psi) \notag\\
&+ b_0 (\eta, \Psi)_{L_g^2(\mathbb{R}^+; H_0^1)} - m b_0 (\theta, \Psi) = \rho_2 b_0(f_6, \Psi).
	\end{align}
	First, we estimate the term $(\psi, \Psi)_{L_g^2(\mathbb{R}^+; H_0^1)}$. Using \eqref{Resolvent equations-2}$_9$, we obtain
	\begin{align*}
		(\psi, \Psi)_{L_g^2(\mathbb{R}^+; H_0^1)} &= (\psi, i\lambda \eta + \eta_s - f_9)_{L_g^2(\mathbb{R}^+; H_0^1)} \\
		&= -i\lambda (\psi, \eta)_{L_g^2(\mathbb{R}^+; H_0^1)} + (\psi, \eta_s)_{L_g^2(\mathbb{R}^+; H_0^1)} - (\psi, f_9)_{L_g^2(\mathbb{R}^+; H_0^1)} \\
		&= -i\lambda (\psi, \eta)_{L_g^2(\mathbb{R}^+; H_0^1)} - \int_0^\infty g'(s)(\psi_x, \eta_x) ds - (\psi, f_9)_{L_g^2(\mathbb{R}^+; H_0^1)}.
	\end{align*}
	By H\"older's and Young's inequalities, using estimates \eqref{g} and \eqref{g'}, we obtain
	\begin{align*}
		|(\psi, \Psi)_{L_g^2(\mathbb{R}^+; H_0^1)}| &\leq |\lambda| \|\psi\|_{L_g^2(\mathbb{R}^+; H_0^1)}\|\eta\|_{L_g^2(\mathbb{R}^+; H_0^1)}\\
 &\quad- \int_0^\infty g'(s) \|\eta_x\| ds \|\psi_x\| + \|\psi\|_{L_g^2(\mathbb{R}^+; H_0^1)} \|f_9\|_{L_g^2(\mathbb{R}^+; H_0^1)} \\
		&\leq |\lambda| b_0^{\frac{1}{2}} \|\psi_x\| \|\eta\|_{L_g^2(\mathbb{R}^+; H_0^1)} + g(0)^{\frac{1}{2}} \left( \int_0^\infty (-g'(s)) \|\eta_x\|^2 ds \right) ^{\frac{1}{2}} \|\psi_x\| \\
&\quad+  b_0^{\frac{1}{2}}\|\psi_x\|\|f_9\|_{L_g^2(\mathbb{R}^+; H_0^1)}\\
		&\leq |\lambda| \left( \epsilon \|U\|_{\mathcal{H}_2}^2 + C_{\epsilon} \|F\|_{\mathcal{H}_2}^2 \right).
	\end{align*}
	Second, we estimate the term $(\eta, \Psi)_{L_g^2(\mathbb{R}^+; H_0^1)}$. From equation \eqref{Resolvent equations-2}$_5$, it follows that
	\begin{align*}
		|(\eta, \Psi)_{L_g^2(\mathbb{R}^+; H_0^1)}| &= |(\eta, i\lambda\psi - f_5)_{L_g^2(\mathbb{R}^+; H_0^1)}| \\
		&\leq \|\eta\|_{L_g^2(\mathbb{R}^+; H_0^1)} \|i\lambda\psi - f_5\|_{L_g^2(\mathbb{R}^+; H_0^1)} \\
		&\leq \|\eta\|_{L_g^2(\mathbb{R}^+; H_0^1)} (|\lambda| \|\psi\|_{L_g^2(\mathbb{R}^+; H_0^1)} + \|f_5\|_{L_g^2(\mathbb{R}^+; H_0^1)}) \\
		&= b_0^{\frac{1}{2}} |\lambda| \|\eta\|_{L_g^2(\mathbb{R}^+; H_0^1)} \|\psi_x\| + b_0^{\frac{1}{2}} \|\eta\|_{L_g^2(\mathbb{R}^+; H_0^1)} \|f_{5x}\| \\
		&\leq |\lambda| \left( \epsilon \|U\|_{\mathcal{H}_2}^2 + C_{\epsilon} \|F\|_{\mathcal{H}_2}^2 \right).
	\end{align*}
	Returning to equation \eqref{eq:Psi}, we get $\|\Psi\|^2 \leq \epsilon\|U\|_{\mathcal{H}_2}^2 + C_{\epsilon}\|F\|_{\mathcal{H}_2}^2$ for $|\lambda| \geq 1$.
\end{proof}

\begin{lemma}\label{boundedness:psix}
	Under Assumption \ref{assumg}, for any $\epsilon>0$, there exists a constant $C_{\epsilon}>0$, independent of $\lambda$, such that
	\begin{align}
		\|\psi_x\|^2 \leq \epsilon\|U\|_{\mathcal{H}_2}^2 + C_{\epsilon}\|F\|_{\mathcal{H}_2}^2
	\end{align}
	for $|\lambda| \geq 1$.
\end{lemma}
\begin{proof}
	From equation \eqref{Resolvent equations-2}$_5$, we get $\Psi_x = i\lambda \psi_x - f_{5x}$. Substituting $\Psi_x$ into the term $\dt b (\psi_x, \Psi_x)$ of equation \eqref{hh}, we obtain
	\begin{align}\label{eqpsix}
		i\lambda\rho_2 \|\Psi\|^2 &- i\lambda \dt b \|\psi_x\|^2 - \dt b (\psi_x,f_{5x})+ \kappa(\varphi_x + \psi, \Psi)\notag \\
&+ (\eta,\Psi)_{L_g^2(\mathbb{R}^+;H_0^1)} - m(\theta, \Psi) = \rho_2 (f_6, \Psi).
	\end{align}
	From the proof of Lemma \ref{boundedness:Psi}, we obtain $$|(\eta, \Psi)_{L_g^2(\mathbb{R}^+; H_0^1)}| \leq |\lambda| \left( \epsilon \|U\|_{\mathcal{H}_2}^2 + C_{\epsilon} \|F\|_{\mathcal{H}_2}^2 \right).$$ Combined with Lemmas \ref{boundedness:ux} and \ref{boundedness:Psi}, from equation \eqref{eqpsix}, it follows that
	\begin{align*}
		\|\psi_x\|^2 \leq \epsilon\|U\|_{\mathcal{H}_2}^2 + C_{\epsilon}\|F\|_{\mathcal{H}_2}^2
	\end{align*}
	for $|\lambda| \geq 1$.
\end{proof}

Consequently, from Lemmas \ref{boundedness:v,q,eta}--\ref{boundedness:psix}, we derive $$\|U\|_{\mathcal{H}_2}=\|(i\lambda I-\mathcal{A}_2)^{-1}F\|_{\mathcal{H}_2} \leq C\|F\|_{\mathcal{H}_2}.$$The proof of this section is complete.

\begin{proof}[Proof of Theorem \ref{exponential stability}]
	By Theorem \ref{Stability Theorem}, Theorem \ref{exponential stability} is a direct consequence of Steps 1 and 2. Moreover, combining the uniform boundedness of the resolvent with the dissipativity estimate in Lemma \ref{dissipative-2}, the decay rate can be taken as $\beta = \min\{\gamma, \sigma, \kappa_1/2\} / C$ for some constant $C > 0$ depending only on the structural parameters.
\end{proof}

\noindent\textbf{Acknowledgement} This work is supported by NSFC (12671232).

\noindent\textbf{Declaration of competing interest} The authors declare that they have no known competing financial interests or personal relationships that could have appeared to influence the work reported in this paper.

\noindent\textbf{Data availability} The authors have nothing to report.

\bibliographystyle{elsarticle-num}

\end{document}